\documentclass[12pt,a4paper]{article}
\usepackage[utf8x]{inputenc}
\usepackage{ucs}
\usepackage[toc,page]{appendix}
\usepackage[top=1in, bottom=1in, left=.75in, right=.75in]{geometry}
\usepackage{amsmath}
\usepackage{color}
\usepackage[nottoc, notlof, notlot]{tocbibind}  
\usepackage{amsfonts,dsfont}
\usepackage{amssymb}
\usepackage{mathtools} 
\usepackage{amsthm}
\usepackage{mathrsfs} 
\usepackage{graphicx}
\usepackage{enumerate}  
\author{Xinxin Chen, Xiaolin Zeng}
\newtheorem{lem}{Lemma}
\newtheorem{thm}{Theorem}

\newtheorem{coro}{Corollary}
\newtheorem{prop}{Proposition}
\newtheorem{rmk}{Remarks}

\usepackage{thmtools}
\usepackage{hyperref}

\declaretheoremstyle[notefont=\bfseries,notebraces={}{},%
    headpunct={},postheadspace=1em]{mystyle}
\declaretheorem[style=mystyle,numbered=no,name=Theorem]{thm-hand}

\def\P{\mathbf{ P}}
\def\E{\mathbf{ E}}
\def\p{\mathbb{ P}}
\def\e{\mathbb{ E}}
\def\Q{\mathbf{ Q}} 
\newcommand{\parent}[1]{\overset{\longleftarrow}{#1}}

\newcommand{\grandpa}[1]{\overset{\Longleftarrow}{#1}}
\graphicspath{{./images/} }  
\title{Speed of Vertex reinforced jump process on Galton-Watson trees}

\date{}
\begin{document}
\maketitle{}
\begin{abstract}
We give an alternative proof of the fact that the vertex reinforced jump process on Galton-Watson tree has a phase transition between recurrence and transience as a function of \(c\), the initial local time, see~\cite{basdevant2012continuous}. Further, applying techniques in~\cite{aidekon2008transient}, we show a phase transition between positive speed and null speed for the associated discrete time process in the transient regime.
\end{abstract}
\section{Introduction and results}
Let \(\mathcal{G}=(V,E)\) be a locally finite graph endowed with its vertex set \(V\) and edge set \(E\). Assign to each edge \(e=\{u,v\}\in E\) a positive real number \(W_e=W_{u,v}\) as its conductance, and assign to each vertex \(u\) a positive real number \(\phi_u\) as its initial local time. Define a continuous-time \(V\) valued process \((Y_t; t\geq0)\) on \(\mathcal{G}\) in the following way: At time \(0\) it starts at some vertex \(v_0\in V\); If \(Y_t=v\in V\), then conditionally on \(\{Y_s; 0\leq s\leq t\}\), the process jumps to a neighbor \(u\) of \(v\) at rate \(W_{v,u}L_u(t)\) where
\begin{equation}
\label{vrjpGW-eq-localtimeplus}
L_u(t):=\phi_u+\int_0^{t} \mathbf{1}_{\{Y_s=u\}}ds.
\end{equation}
We call \((Y_t)_{t\geq0}\) the vertex reinforced jump process (VRJP) on \((\mathcal{G},W)\) starting from \(v_{0}\).

It has been proved in~\cite{davis2002continuous} that when \(\mathcal{G}=\mathbb{Z}\), \((Y_{t})\) is recurrent. When \(\mathcal{G}=\mathbb{Z}^d\) with \(d\geq 2\), the complete description of its behavior has not been revealed even though lots of effort has been made, see e.g.~\cite{angel2014localization,basdevant2012continuous,collevecchio2009limit,davis2002continuous,davis2004vertex,sabot2011edge}.

Here we are interested in the case when \(\mathcal{G}\) is a supercritical Galton-Watson tree, as we will see, acyclic property of trees largely reduces the difficulty to study this model. In~\cite{collevecchio2009limit} it is shown that the VRJP (for constant parameters \(W_{v,u}\equiv 1\) and \(\phi_{x}\equiv 1\)) on 3-regular tree has positive speed and satisfies a central limit theorem. Later, Basdevant and Singh~\cite{basdevant2012continuous} gave a precise description of the phase transition of recurrence/transience for VRJP on supercritical Galton-Watson trees. In this paper, our main results, Theorem~\ref{vrjpGW-speed}, describes the ballistic case of the VRJP when it is transient on supercritical Galton-Watson trees without leaves. Our proof is based on the random walk in random environment (RWRE) representation result of Sabot and Tarr\`es~\cite{sabot2011edge}, and on techniques in the studies of RWRE on trees, especially on a result of Aidekon~\cite{aidekon2008transient} (see also e.g.\cite{hu2007subdiffusive,hu2007slow} for more on the studies of RWRE on trees).

Consider a rooted Galton-Watson tree \(T\) with offspring distribution \((q_k, k\geq0)\) such that
\[
b:=\sum_{k\geq0}kq_k>1.
\]
For some constant \(c>0\), we denote VRJP\((c)\) the process \((Y_{t})\) on the Galton-Watson tree \(T=(V,E)\) with \(W_e\equiv 1\), \(\forall e\in E\) and \(\phi_x\equiv c\), \(\forall x\in V\), starting from the root \(\rho\). Hence the behaviors of this process depends on  \(\mathcal{G}\) and \(c\).  This definition is equivalent to VRJP with constant edge weight \(W\) and initial local time \(1\), up to a time change. We first recall the phase transition result obtained in~\cite{davis2002continuous}. Let \(A\) be an inverse Gaussian distribution of parameters \((1,c^2)\), i.e.
\begin{equation}
\label{vrjpGW-eq-A}
\P(A\in dx)=\mathds{1}_{x\geq 0}\frac{c}{\sqrt{2\pi x^3} }\exp\Big\{-\frac{c^2(x-1)^2}{2x}\Big\}dx,
\end{equation}
 The expectation w.r.t.\ \(\P(dx)\) is denoted \(\E\). 
\begin{thm}[Basdevant \& Singh]
\label{vrjpGW-thm-phasetransition}
  Let \(\mu(c)=\inf_{a\in\mathbb{R}}\E[A^a]=\E[\sqrt{A}]\), then the VRJP\((c)\) on a supercritical GW tree with offspring mean \(b\) is recurrent a.s.\ if and only if \(b\mu(c)\leq1\).
\end{thm}
\begin{rmk}
This phase transition was proved in~\cite{basdevant2012continuous} by considering the local times of VRJP\@. We will give another proof from the point of view of a random walk in random environment (RWRE), as a consequence of Theorem~\ref{vrjpGW-mixture_thm}.
\end{rmk}

When \(b\mu(c)>1\), a further question is to study the escape rate of the process. Define the speed of the process \((Y)\) by
\begin{equation}\label{vrjpGW-vY}
v(Y):=\liminf_{t\rightarrow\infty}\frac{d(\rho, Y_t)}{t}=\lim_{t\rightarrow\infty}\frac{d(\rho, Y_t)}{t}
\end{equation}
where \(d\) is the graph distance, and the last equality will be justified by Lemma~\ref{vrjpGW-iid_reg_time}. To study the speed, we use the RWRE point of view, relying on a result of Sabot \& Tarrès~\cite{sabot2011edge}, in particular, on the following fact:

Let \((Y_{t})\) be a VRJP on a finite graph \(\mathcal{G}=(V,E)\) with edge weight \((W)\) and initial local time \((\phi)\). If \((Z_{t})\) is defined by
\begin{equation}
\label{vrjpGW-eq-def-Z}
Z_t:=Y_{D^{-1}(t)} \text{ where }
D(t):=\sum_{x\in {V}}(L_x(t)^2-\phi_x^2),
\end{equation}
then \((Z_{t})\) is a mixture of Markov jump processes (c.f.\ also~\cite{sabot2013ray}). Moreover, the mixing measure is explicit.

Applying this result to our VRJP\((c)\) on a tree, denote \((\eta_{n})_{n\geq 0}\) the discrete time process associated to \((Z_{t})\), it turns out that \((\eta_{n})\) is a random walk in random environment. In~\cite{aidekon2008transient}, Aidekon gave a sharp and explicit criterion for the asymptotic speed to be positive, for random walks in random environment on Galton-Watson trees such that the environment is site-wise independent and identically distributed. This result cannot apply directly to the time changed VRJP\((c)\), since the quenched transition probability depends also on the environment of the neighbors, see~\eqref{vrjpGW-eq-transition-prob}.

Aidekon's idea was to say that, the slowdown comes from wanders in long pipes, therefore, the random walk is roughly a trapped random walk. To study the speed, it is enough to look at the random walk on the traps, that is, pipes. This also explains why the criterion depends on \(q_{1}\), the probability that the GW tree generates one offspring.

In our case the environment is almost i.i.d., the same idea will also work. Compare to~\cite{aidekon2008transient}, we mainly deal with the local dependences of the quenched probability transition. We believe that same type of criterion also holds for a larger type of random walk in random environment, with suitable conditions on the moments of the environment and locality of the transition probabilities.

Let us state our criterion, similar to~\eqref{vrjpGW-vY}, define
\begin{equation}
  \label{vrjpGW-eq-veta-vz}
  v(Z)=\liminf_{t\to\infty}\frac{d(\rho,Z_{t})}{t},\ v(\eta)=\liminf_{n\to \infty}\frac{d(\rho,\eta_{n})}{n}.
\end{equation}
To study the speed, our techniques can only deal with trees without leaves, hence we assume that \(q_0=0\). In addition, we assume that
\[
M:=\sum_{k\geq0}k^2q_k<\infty.
\]
For any \(r\in\mathbb{R}\), let
\[
\xi_r=\xi_r(c):=\E[A^{-r}].
\]
By~\eqref{vrjpGW-eq-A}, \(\xi_r\in(0,\infty)\) for any \(r\). In particular, \(\mu(c)=\xi_{-1/2}(c)\). Our main theorem states that the speed depends on the value of \(q_1\) and \(c\).
\begin{thm}
  \label{vrjpGW-speed}
Consider VRJP\((c)\) on a supercritical GW tree such that  \(b\mu(c)>1\), we have
\begin{enumerate}[(1)]
\item \(\lim_{t\to\infty}\frac{d(\rho,Z_{t})}{t}\) and \(\lim_{n\to \infty}\frac{d(\rho,\eta_{n})}{n}\) exist almost surely,
\item Assume \(q_{0}=0\)  and \(M<\infty\). If \(q_{1}\xi_{1/2}<1\), then \(v(\eta)>0,\ v(Z)>0\); if \(q_{1}\xi_{1/2}>1\), then \(v(\eta)=v(Z)=0\).
\end{enumerate}
\end{thm}

\begin{rmk}
Let $\Lambda$ be the Lebesgue measure of $\{t\in\mathbb{R}: \E[A^{2t}]<1/q_1\}$, $q_1\xi_{1/2}<1$ is equivalent to the condition $\Lambda>1$. In other words, same as in Theorem 1.1 of \cite{aidekon2008transient} if $\Lambda>1$, the walk has positive speed; if $\Lambda<1$, its speed is null.
\end{rmk}

\begin{coro}\label{vrjpGW-coro-speed}
VRJP\((c)\) \((Y_t)_{t\geq 0}\) on a supercritical GW tree such that \(b\mu(c)>1\),  admits a speed \(v(Y)\geq0\) a.s. If in addition \(q_0=0,\ M<\infty\) and \(q_1\xi_{1/2}<1\), then 
 \(v(Y)>0\).
 \end{coro}
\begin{rmk}
  Our method cannot tackle the critical case \(q_{1}\xi_{1/2}=1\). Moreover, whether \(q_{1}\xi_{1/2}>1\) implies \(v(Y)=0\) remains unknown, since we do not have estimates on the random time change.
\end{rmk}
The rest of this paper is organized as follows. In Section~\ref{vrjpGW-rwre}, we use a result of Sabot \& Tarres~\cite{sabot2011edge} to recover the RWRE structure of VRJP\@. Section 3 is devoted to an alternative proof of Theorem~\ref{vrjpGW-thm-phasetransition}, as an application of the RWRE point of view. Section 4 establishes the existence of the speed for the RWRE and prove Theorem~\ref{vrjpGW-speed}. The proofs of some technical lemmas are left in Appendix.
\section{RWRE on Galton-Watson tree}
\label{vrjpGW-rwre}
\subsection{Mixture of Markov jump process by changing times}
In this subsection, we consider a VRJP \((Y_t)_{t\geq0}\) on a tree \(T=(V,E)\) rooted at \(\rho\), with edge weights \((W)\) and initial local time \((\phi)\). If \(x\neq \rho\), let \(\parent{x}\) be the parent of \(x\) on the tree, the associated edge is denoted by \(e_x=(x,\parent{x})\) with weight \(W_{e_x}\).

Recall that the time changed version of VRJP \((Z_{t})\) defined in~\eqref{vrjpGW-eq-def-Z} is mixture of Markov jump processes with correlated mixing measure. The advantage of considering VRJP on trees is that, the random environment becomes independent.
\begin{thm}
\label{vrjpGW-mixture_thm}
 Let \(T=(V,E)\) be a tree rooted at \(\rho\), endowed with edge weights \((W_{e})_{e\in E}\) and initial local times \((\phi_{x})_{x\in V}\). Let \((A_x,x\in V\setminus\{ \rho\})\) be independent random variables defined by
\[\P(A_x\in da)=\mathds{1}_{\mathbb{R}^+}(a)\phi_{{x}}\sqrt{\frac{{W}_{e_x}\phi_{x}\phi_{\parent{x}}}{2\pi a^3}}\exp(-{W}_{e_x}\phi_x\phi_{\parent{x}}\frac{(a-1)^2}{2a})da.\]
If \(X_t\) is a mixture of Markov jump processes starting from \(\rho\), such that, conditionally on \((A_x,x\in V\setminus\{ \rho\})\), \(X_t\) jumps from \(x\) to \(\parent{x}\) at rate \(\frac{1}{2}{W}_{e_x} \frac{\phi_{\parent{x}}}{\phi_xA_x}\) and  from \(\parent{x}\) to \(x\) at rate \(\frac{1}{2}{W}_{e_x} \frac{\phi_xA_x}{\phi_{\parent{x}}}\). Then \(X_t\) and \(Z_t\) (defined in~\eqref{vrjpGW-eq-def-Z}) has the same distribution.
\end{thm}
\begin{proof}
  On trees, VRJP observed at times when it stays on any finite sub-tree \(T_f=(V_{f},E_{f})\) (also rooted at \(\rho\)) of \({T}\), behaves the same as VRJP restricted to \({T}_f\); moreover, the restriction is independent of the VRJP outside \({T}_f\). Therefore, it is enough to prove the theorem on finite tree \({T}_f\). By Theorem 2 of~\cite{sabot2011edge} (with a slight modification of the initial local time, or a more detailed version in~\cite{2015arXiv150704660S}, appendix B), if we denote
\[l_{x}(t)=\int_{0}^{t}\mathds{1}_{Z_{s}=x}ds,\]
then
\[U_{x}=\frac{1}{2}\lim_{t\to \infty}\left(\log \frac{l_{x}(t)+\phi_{x}^{2}}{l_{\rho}(t)+\phi_{\rho}^{2}}-\log \frac{\phi_{x}^{2}}{\phi_{\rho}^{2}} \right)\]
exists a.s.\ and \(\{U_x: x\in {V}_f,\}\) has distribution (where \(du=\prod_{x\neq \rho}du_{x}\))
\[d \mathcal{Q}_{\rho,{T}_f}^{W,\phi}(u)=\frac{\mathds{1}_{u_{\rho}=0}\prod_{x\neq \rho}\phi_x}{\sqrt{2\pi}^{|V_{f}|-1}}e^{-\sum_{x\in V_{f}} u_x -\sum_{\{x,y\}\in E_f}\frac{1}{2}{W}_{x,y}(e^{u_x-u_y}\phi_y^2+e^{u_y-u_x}\phi_x^2-2\phi_x\phi_y)}\sqrt{\prod_{\{x,y\}\in E_f}{W}_{x,y}e^{u_x+u_y}}du.\]
Now, conditionally on \((U_{x})\), \(Z_t\) is a Markov process which jumps at rate (from \(x\) to \(z\)) \(\frac{1}{2}{W}_{x,z}e^{U_z-U_x}\). For \(e_{x}=(x,\parent{x})\in T_f\), if we apply the change of variable \(y_{e_{x}}=(u_{\parent{x}}-\log \phi_{\parent{x}})-(u_x-\log\phi_x)\), then (note that \(u\mapsto y\) is a diffeomorphism and \(dy=du\)) the density of \((y)\) writes
\[d\mathcal{Q}_{\rho,T_f}^{W,\phi}(u)=\prod_{e_{x}=\{x,\parent{x}\}\in E_f}\sqrt{\frac{{W}_{e_{x}}\phi_x\phi_{\parent{x}}}{2\pi}}\exp\left(\frac{1}{2}(y_{e_{x}}-{W}_{e_{x}}\phi_x\phi_{\parent{x}}(e^{y_{e_{x}}}+e^{-y_{e_{x}}}-2))\right)dy.\]
Plugging \(a_x=e^{-y_{e_{x}}}\) entails that \(a_{x}\) is Inverse Gaussian distributed with parameter \((1,W_{e_{x}}\phi_x\phi_{\parent{x}})\) i.e.
\[d\mathcal{Q}_{\rho,T_f}^{W,\phi}(a)=\prod_{x\in V_f\setminus \{\rho\}}\mathds{1}_{a_{x}>0}\sqrt{\frac{{W}_{e_{x}}\phi_x\phi_{\parent{x}}}{2\pi a_x^3}}\exp(-{W}_{e_{x}}\phi_x\phi_{\parent{x}}\frac{(a_x-1)^2}{2a_x})da_{x}\]
Finally note that
\[\frac{1}{2}{W}_{x,z}e^{u_z-u_x}=
\begin{cases}
\frac{1}{2}{W}_{x,z}\frac{\phi_z}{\phi_xa_x} & \text{ if }z=\parent{x}\\
   \frac{1}{2}{W}_{x,z} \frac{\phi_za_z}{\phi_x} & \text{ if }\parent{z}=x.
\end{cases}
\]
\end{proof}
For VRJP\((c)\) on a GW tree, the theorem immediately implies:
\begin{coro}
  \label{vrjpGW-coro-rwre-on-gwtree-vrjpc}
On a sampled GW tree \(T=(V,E)\), the time changed VRJP\((c)\) \((Z_{t})\) is a random walk in  environment given by \((A_{x},x\in V\setminus \{\rho\})\), where \((A_{x})\) are i.i.d.\ inverse Gaussian distributed with parameters \((1,c^{2})\), and conditionally on the environment, the process jumps at rate
\begin{equation}
\label{vrjpGW-eq-jump-rate-Z-vrjpc}
\begin{cases}
  \frac{1}{2A_{x}} & \text{from \(x\) to \(\parent{x}\)}\\
  \frac{1}{2}A_x & \text{ from \(\parent{x}\) to \(x\)}.
\end{cases}
\end{equation}
\end{coro}
\subsection{RWRE on Galton Watson tree and notations}

In the sequel, let \(T=(V,E)\) be a Galton-Watson tree with offspring distribution \(\{q_k; k\geq0\}\). Recall that \((\eta_{n})_{n\geq 0}\) denotes the discrete time process associated to \((Z_{t})\) (or \((Y_{t})\)), which is a random walk in random environment.

Note that there are two levels of randomnesses in the environment. First, we sample a GW tree, \(T\), whose law is denoted by \(GW(dT)\). Then, given the tree \(T\) (rooted at \(\rho\)), we define \(\omega=\{A_x, x\in V\setminus\{\rho\}\}\) as in Corollary~\ref{vrjpGW-coro-rwre-on-gwtree-vrjpc}, whose law is  \(\prod_{x\in T\setminus \{\rho\}}\P(d A_{x})\), which we denote abusively \(\P(d\omega)\). Finally, given \((w, T)\), the Markov jump process \((Z_t; t\geq0)\) is defined by its jump rate in~\eqref{vrjpGW-eq-jump-rate-Z-vrjpc}.

For convenience, we artificially add a vertex \(\parent{\rho}\) to \(T\), designing the parent of the root. Let \(A_\rho\) be another copy of \(A\), independent of all others. Now, (abusively) let \(\omega=(A_x, x\in V)\) be the enlarged environment. Given \((\omega, T)\), define the new Markov chain \(\eta\), which is a random walk on \(V\cup\{\parent{\rho}\}\), with transition probabilities
\begin{equation}
\label{vrjpGW-eq-transition-prob}
\begin{cases}
  p(x,\parent{x})\propto 1/A_{x}\\ 
  p(x,z)\propto A_{z} & \text{where }\parent{z}=x\in V \\
  p(\parent{\rho},\rho)=1
\end{cases}
\end{equation}
This modification will not change the recurrence/transience behavior of the RWRE \(\eta\) nor its speed in the transient regime. We will always work with this modification in the sequel.

Let us now introduce the notation of quenched and annealed probabilities. Given the environment \((\omega, T)\), let \(P_x^{\omega, T}\) denote the quenched probability of the random walk \(\eta\) with \(\eta_0=x\in V\) a.s. Denote by \(\p_x^T\), \(\Q\), \(\p_\rho\) the mesures:
\begin{align*}
\p_x^T(\cdot)&:=\int P_x^{\omega, T}(\cdot)\P(d\omega),\\
\Q(\cdot)&:=\int \mathds{1}_{\{\cdot\}}\P(d\omega)GW(dT)\\
\p_\rho(\cdot) &:=\int \p_\rho^T(\cdot) GW(dT),
\end{align*}
and the associated expectations are denoted by \(E_x^{\omega, T}\), \(\e_x^T\), \(\e_{\Q}\) and \(\e\). For brevity, we omit the starting point if the random walk starts from the root; that is, we write \(P^{\omega, T}\), \(\p^T\) and \(\p\) for \(P_\rho^{\omega, T}\), \(\p_\rho^T\) and \(\p_\rho\),  Notice that \(\p\) is the annealed law of \(\eta\).

For any vertex \(x\), let \(|x|=d(\rho,x)\) be the generation of \(x\) and denote by \([\![ \rho, \, x]\!]\) the unique shortest path from \(x\) to the root \(\rho\), and \(x_i\) (for \(0\le i\le |x|\)) the vertices on \([\![ \rho, \, x]\!]\) such that \(|x_i|=i\). In particular, \(x_0=\rho\) and \(x_{|x|}=x\). In words, \(x_i\) (for \(i<|x|\)) is the ancestor of \(x\) at generation \(i\). Also denote \(\, ]\! ] \rho, \, x]\!] := [\![ \rho, \, x]\!] \backslash \{ \rho\}\) and \(\, ]\! ] \rho, \, x[\![ := [\![ \rho, \, x]\!] \backslash \{ \rho, x\}\). Moreover, for \(u,v\in T\), we write \(u<v\) if \(u\) is an ancestor of \(v\).

\section{Phase transition: an alternative proof of Theorem~\ref{vrjpGW-thm-phasetransition}}
\label{vrjpGW-transition}
The ideas follow from Lyons and Pemantle~\cite{lyons1992random}, by means of random electrical network.
\begin{proof}[Proof of Theorem~\ref{vrjpGW-thm-phasetransition}]
The RWRE is equivalent to an electrical network with random conductances:
\[
C_{e_x}:=C(x,\parent{x})=(\prod_{u\in \, ]\! ] \rho, \, x[\![} A_u)^2A_x, \forall x\in V\setminus\{\rho\}.
\]
We omit the proof of the transient case which is quite similar to that in Lyons and Pemantle~\cite{lyons1992random}, however, we will detail the recurrence case. That is, we will show that if \(b\mu(c)\leq 1\), then the RWRE is recurrent a.s.

First consider the case \(b\mu(c)<1\), note that
\begin{align*}
\e_\Q\Big[\sum_{n\geq 1} \sum_{|x|=n}C_{e_x}^{1/4}\Big]&= \sum_{n\geq 1}\int\Big(\int \sum_{|x|=n}C_{e_x}^{1/4} \P(d\omega)\Big) GW(dT)\\
&= \sum_{n\geq 1}\int\sum_{|x|=n} \E[A^{1/2}]^{n-1} \E[A^{1/4}]GW(dT)\\
&= \sum_{n\geq 1}b^n  \E[A^{1/2}]^{n-1} \E[A^{1/4}].
\end{align*}
Because \(\mu(c)=\E[A^{1/2}]<1/b\), we have, for some constants \(c_{1},c_{2}\in \mathbb{R}^{+}\)
\[
\e_\Q\Big[\sum_{n\geq 1} \sum_{|x|=n}C_{e_x}^{1/4}\Big]\leq c_1 \sum_{n\geq 0} (b\mu(c))^n\leq c_2<\infty,
\]
which implies that 
\[
\sum_{n\geq 1} \sum_{|x|=n}C_{e_x}^{1/4}<\infty, \ \Q\text{-a.s.}
\]
As a result, there exists a stationary probability a.s., moreover \(\eta\) is positive recurrent.

Turning to the case \(b\mu(c)=1\), let \(\Pi_n:=\{e_x: |x|=n\}\) be a sequence of cutsets. Observe that 
\[
W_n:=\sum_{|x|=n}\prod_{u\in \, ]\! ] \rho, \, x]\!]} A_u^{1/2} =\sum_{|x|=n} C_{e_x}^{1/4} A_x^{1/4}.
\] is a martingale with respect to \(\mathcal{F}_{n}=\sigma(\{A_{x}, |x|\le n\})\). By Biggin's theorem~(\cite{biggins1977martingale,lyons1997simple}), it converges a.s.\ to zero. More precisely, we use the equivalent condition (iv) in page 2 of \cite{lyons1997simple}, since \(t\mapsto \E(A^{t})\) attain its minimum at \(t=\frac{1}{2}\), in terms of \(m(\alpha), m'(\alpha)\) in \cite{lyons1997simple}, we actually have \(\alpha=-\frac{1}{2}\) and \(m(\alpha)=1,\ m'(\alpha)=0\), therefore, the criterion (iv) is never satisfied.

We are going to show that \(\Q\)-a.s.,
\begin{equation}\label{vrjpGW-recnul}
\liminf_{n\rightarrow\infty} \sum_{|x|=n} C_{e_x}^{1/4}=0,
\end{equation}
in particular, this will imply that \(\Q\)-a.s.\ \(\inf_{\Pi: \textrm{ cutset}} \sum_{e_x\in \Pi}C_{e_x}=0\). By the trivial half of the max-flow min-cut theorem, the corresponding network admits no flow a.s. Hence, the random walk is a.s. recurrent. Observes that
\begin{align*}
 \sum_{|x|=n} C_{e_x}^{1/4}&=\sum_{|x|=n}\prod_{u\in \, ]\! ] \rho, \, x[\![} A_u^{1/2}A_x^{1/4}1_{\{A_x\geq 1\}}+\sum_{|x|=n}\prod_{u\in \, ]\! ] \rho, \, x[\![} A_u^{1/2}A_x^{1/4}1_{\{A_x< 1\}}\\
 &= \sum_{|x|=n}\prod_{u\in \, ]\! ] \rho, \, x]\!]} A_u^{1/2} A_x^{-1/4}1_{\{A_x\geq 1\}}+ \sum_{|y|=n-1} \prod_{u\in \, ]\! ] \rho, \, y]\!]} A_u^{1/2}\sum_{x: \parent{x}=y}A_x^{1/4}1_{\{A_x< 1\}}\\
 &\leq W_n+\sum_{|y|=n-1}\prod_{u\in \, ]\! ] \rho, \, y]\!]} A_u^{1/2} \nu_y,
\end{align*}
where \(\nu_y\) denotes the number of children of \(y\). Letting \(n\) go to infinity yields that
\begin{align*}
0\leq \liminf_{n\rightarrow\infty}\sum_{|x|=n} C_{e_x}^{1/4}&\leq \liminf_{n\rightarrow\infty}\sum_{|y|=n-1}\prod_{u\in \, ]\! ] \rho, \, y]\!]} A_u^{1/2} \nu_y.
\end{align*}
For any \(K\geq1\), separating the sum over vertices \(y\) according to \(\{\nu_y< K\}\) or \(\{\nu_y\geq K\}\), the last term is bounded by
\begin{align*}
&\lim_{n\rightarrow\infty}KW_{n-1}+ \liminf_{n\rightarrow\infty}\sum_{|y|=n-1}\prod_{u\in \, ]\! ] \rho, \, y]\!]} A_u^{1/2}\nu_y1_{\{\nu_y\geq K\}}\\
=&\liminf_{n\rightarrow\infty}\sum_{|y|=n-1}\prod_{u\in \, ]\! ] \rho, \, y]\!]} A_u^{1/2}\nu_y1_{\{\nu_y\geq K\}}.
\end{align*}
By Fatou's lemma,
\begin{align*}
&\e_\Q\Big(\liminf_{n\rightarrow\infty}\sum_{|y|=n-1}\prod_{u\in \, ]\! ] \rho, \, y]\!]} A_u^{1/2}\nu_y1_{\{\nu_y\geq K\}}\Big)\\
\leq&\liminf_{n\rightarrow\infty} \e_\Q\Big(\sum_{|y|=n-1}\prod_{u\in \, ]\! ] \rho, \, y]\!]} A_u^{1/2}\nu_y1_{\{\nu_y\geq K\}}\Big)
=\e_\Q[\nu_\rho, \nu_\rho\geq K],
\end{align*}
since for all \(|y|=n-1\), \(\nu_y\) is independent of \(\prod_{u\in \, ]\! ] \rho, \, y]\!]} A_u^{1/2}\)  and \(\e_\Q\Big(\sum_{|y|=n-1}\prod_{u\in \, ]\! ] \rho, \, y]\!]} A_u^{1/2}\Big)=1\).
Consequently, for any \(K\geq1\),
\begin{align*}
\e_\Q\Big[\liminf_{n\rightarrow\infty}\sum_{|x|=n} C_{e_x}^{1/4}\Big]&\leq \e_\Q[\nu_\rho, \nu_\rho\geq K].
\end{align*}
As \(b=\e_\Q[\nu_y]<\infty\), letting \(K\rightarrow\infty\) gives 
\[
\e_\Q\Big[\liminf_{n\rightarrow\infty}\sum_{|x|=n} C_{e_x}^{1/4}\Big]=0.
\]
This implies~\eqref{vrjpGW-recnul}.
\end{proof}
\section{Speed when transient}
\label{vrjpGW-sec-speed}
Turning to the positivity of \(v(Z)\) and \(v(\eta)\), note that the processes \((Z_{t})\) and \((\eta_{n})\) are mixture of Markov processes but \((Y_{t})\) is not, in fact, \((Y_{t})\) escapes faster than \((Z_{t})\), in particular, when \(v(Z)>0\), we have \(v(Y)>0\). But we are not sure whether \(v(Z)=0\) implies \(v(Y)=0\).

\subsection{Regeneration structure}
In this section, we show that, when the process \((\eta_{n})\) (or \((Z_{t})\)) is transient, its path can be cut into independent pieces, using the notion of regeneration time. As a consequence, the speed \(v(\eta)\), \(v(Z)\) exists a.s.\ as a limit (not just a \(\liminf\)).

On a tree, when a random walk traverses an edge for the first and last time simultaneously, we
say it regenerates since it will now remain in a previously unexplored sub-tree. For any vertex \(x\), let \(\mathcal{D}(x)=\inf\{k\geq 1, \ \eta_{k-1}=x,\eta_k=\parent{x}\}\), write \(\tau_n=\inf\{k\geq 0, \ |\eta_k|=n\}\) and define the regeneration time recursively by
\[
\begin{cases}
  \Gamma_0=0\\
  \Gamma_n=\Gamma_n(\eta)=\inf\{k>\Gamma_{n-1};\ d(\eta_k)\geq 3,\mathcal{D}(\eta_k)=\infty,\tau_{|\eta_k|}=k\}
\end{cases}
\]
where \(d(x)\) is the degree of the vertex \(x\).
\begin{lem}
\label{vrjpGW-iid_reg_time}
  Let \(S(\cdot)=\p(\cdot|d(\rho)\geq 3,\ \mathcal{D}(\rho)=\infty)\), if \(\eta\) is transient, then
  \begin{enumerate}[i)]
  \item For any \(n\geq 1\), \(\Gamma_n<\infty\) \(\p\)-a.s.
  \item Under \(\p\), \((\Gamma_{n+1}-\Gamma_n,|\eta_{\Gamma_{n+1}}|-|\eta_{\Gamma_n}|,A_{\Gamma_{n+1}})_{n\geq 1}\) are independent and distributed as \((\Gamma_1,|\eta_{\Gamma_1}|,A_{\Gamma_1})\) under \(S\).
  \item \(E_S(|\eta_{\Gamma_1}|)<\infty\).
  \end{enumerate}
\end{lem}

We feel free to omit the proof because it is analogue to `Fact' in~\cite{aidekon2008transient} p.10. In addition,  Lemma~\ref{vrjpGW-iid_reg_time} also holds without assuming \(d(\eta_k)\geq 3\) in the definition of \(\Gamma_n\), but we will need this assumption later in the proof of Lemma~\ref{vrjpGW-compare_to_Z}.

By strong law of large numbers, one immediately sees that there exist two constants \(c_4\geq c_3\geq1\) such that \(\p\)-a.s.,
\begin{align*}
\lim_{n\rightarrow\infty}\frac{\vert\eta_{\Gamma_{n}}\vert}{n}&=c_3\in [1,\infty),\quad \lim_{n\rightarrow\infty}\frac{\Gamma_{n}}{n}=c_4\in [c_3,\infty].
\end{align*}
In addition, for any \(n\geq1\), there exists a unique \(u(n)\in\mathbb{N}\) such that
\[
\Gamma_{u(n)}\leq n< \Gamma_{u(n)+1}
\]
and \(|\eta_{\Gamma_{u(n)}}|\leq |\eta_n|<|\eta_{\Gamma_{u(n)+1}}|\). Letting \(n\) go to infinity, (in particular \(u(n)\rightarrow\infty\)) in
\[
\frac{|\eta_{\Gamma_{u(n)}}|}{\Gamma_{u(n)+1}}\leq \frac{|\eta_n|}{n}<\frac{|\eta_{\Gamma_{u(n)+1}}|}{\Gamma_{u(n)}}=\frac{|\eta_{\Gamma_{u(n)+1}}|}{u(n)}\frac{u(n)}{\Gamma_{u(n)}}.
\]
We have \(\p\)-a.s.
\[
\frac{|\eta_n|}{n}\rightarrow v(\eta):=\frac{c_3}{c_4}\in[0,1].
\]
For \(Z_t\), the same arguments can be applied. As a consequence of the i.i.d.\ decomposition, \(v(Z)=\lim_{t\rightarrow\infty}\frac{|Z_t|}{t}\) exists a.s.
The existence of \(v(Y)=\lim_{t\rightarrow\infty}\frac{|Y_t|}{t}\) can be justified by performing the time change \(D(t)\) between consecutive regenerative epochs.

\subsection{The auxiliary one dimensional process}
The RWRE can also be defined on the deterministic graph \(\mathbb{H}=\{-1,0,1,\ldots\}\), on which many quantities are viable by explicit computations. The strategy is to compare the random walk on a tree to the random walk on the half line, in the forth coming sections we will explain how these comparisons will be done. In this section we list some properties of the one dimensional random walk, their proofs can be found in Appendix~\ref{vrjpGW-1dappendix}.

Let \(\tilde{\eta}_n\) be the random walk on the half line \(\mathbb{H}=\{-1,0,1,\ldots\}\) in the random environment \(\omega=(A_k, k\geq 0)\) which are i.i.d.\ copies of \(A\) under \(\P\), with transition probability according to~(\ref{vrjpGW-eq-transition-prob}); that is,
\[
\begin{cases}
  p(i,i+1)=\frac{A_{i+1}}{1/A_{i}+A_{i+1}} & i\geq 0\\
  p(i,i-1)=\frac{1/A_{i}}{1/A_{i}+A_{i+1}} & i\geq 0\\
  p(-1,0)=1
\end{cases}
\]
Similarly we denote \( \tilde{P}_{i}^{\omega},\tilde{\p}_i, \tilde{E}_{i}^{\omega},\tilde{\e}_{i}\) respectively the quenched and annealed probability/expectation for such process starting from \(i\), and for any \(n\in \mathbb{H}\), define the following stopping times
\[\tilde{\tau}_{n}=\inf\{k\geq 0,\ \tilde{\eta}_{k}=n\},\ \ \ \tilde{\tau}_{n}^{*}=\inf\{k\geq 1,\ \tilde{\eta}_{k}=n\}.\]
Let \(F_{1},F_{2}>0\) be two expressions which can depend on any variable, but in particular on \(n\). If there exists \(f:\mathbb{N}\to \mathbb{R}^+\) with \(\lim_{n\to \infty}\frac{1}{n}\log f(n)=0\) such that \(F_{1}f(n)\geq F_{2}\), then we denote \(F_1\gtrsim_{n} F_2\) (\(F_1\) greater than \(F_2\) up to polynomial constant).

Recall that \(A\) is Inverse Gaussian distributed with parameter \((1,c^{2})\), define the rate function associated to \(\log A\) by
\begin{equation}
\label{vrjpGW-eq-ratefunc-I}
I(x)=\sup_{t\in \mathbb{R}}\{tx-\log \E(A^t)\},
\end{equation}
also define
\begin{equation}
  \label{vrjpGW-eq-tstar}
  t^{*}=\sup\{t\in \mathbb{R},\ \E(A^t)q_1\leq 1\}.
\end{equation}
Here are the list of estimates in dimension one.
\begin{lem}
\label{vrjpGW-oneMAMA}
For any \(z>0\) and \(0<z_{1}<1\), we have, for any \(0<a<1\)
\[\tilde{\p}_0(\tilde{\tau}_n\wedge \tilde{\tau}_{-1}>m|A_0\in [a,\frac{1}{a}])\gtrsim_{n} \exp\{-n\left(z_1I(\frac{z}{2z_1})+(1-z_1)I(\frac{-z}{2(1-z_1)})\right)\}\]
where \(m\in \mathbb{N}\) is such that \(n=\lfloor \frac{\log m}{z}\rfloor\).
\end{lem}
\begin{lem}
\label{vrjpGW-LDP}
Denote
\[L'=\sup_{z>0,\ 0<z_1<1}\{ \frac{\log q_1}{z}-\frac{z_1}{z}I(\frac{z}{2z_1})-\frac{1-z_1}{z}I(\frac{-z}{2(1-z_1)})\},\]
we have \(L'=-t^{*}+\frac{1}{2}\).
\end{lem}
\begin{lem}
\label{vrjpGW-RWRE1}
Define, for \(i\in \mathbb{H}\) and any stopping time \(\tau\),
\(\tilde{G}^{\tau}(i,i)=\tilde{E}_{i}^{\omega}(\sum_{k=0}^{\tau}\mathds{1}_{\tilde{\eta}_{k}=i})\).
Let \(0\leq Y_{1}<Y_{2}<y<Y_{3}\) be points on the half line, we have, for any \(0\leq\lambda\leq 1\),
\begin{equation}\label{vrjpGW-bdPG}
\widetilde{P}_{Y_{1}}^{\omega}(\tilde{\tau}_{y}<\tilde{\tau}_{{Y_{1}}-1})\widetilde{G}^{\tilde{\tau}_{Y_{1}}\wedge \tilde{\tau}_{Y_{3}}}(y,y)\leq \widetilde{E}^\omega_{Y_1}[\tilde{\tau}_{{Y_1}-1}\wedge \tilde{\tau}_{Y_3}].
\end{equation}
\begin{equation}
    \label{vrjpGW-eqtilTpTm}
   \widetilde{E}^\omega_{Y_1}[\tilde{\tau}_{{Y_1}-1}\wedge \tilde{\tau}_{Y_3}]^\lambda\leq S_{\lambda, [\![  Y_1, Y_2 ]\!]}\bigg(1+A^\lambda_{Y_{2}+1}\Big(1+\widetilde{E}_{Y_{2}+1}^{\omega}[\tilde{\tau}_{Y_{2}}\wedge \tilde{\tau}_{Y_{3}}]^\lambda\Big)\bigg).
  \end{equation}
where
\[
S_{\lambda, [\![  Y_1, Y_2 ]\!]}:=1+2A_{Y_1}^\lambda\sum_{Y_1<z\leq Y_2}\prod_{Y_1<u<z}A_u^{2\lambda}A_z^\lambda+A_{Y_1}^\lambda\prod_{Y_1<u\leq Y_2}A_u^{2\lambda}.
\]
\end{lem}

\begin{lem}\label{vrjpGW-bdtau}
If \(0\leq \lambda<( t^{*}-\frac{1}{2}) \wedge 1\), then there exists sufficiently small \(\delta>0\) such that for all \(n_{1}>0\)
\[
\E\Big((1+\frac{1}{A^\lambda_{n_1}})(1+\frac{1}{A_{n}})A_0\widetilde{E}^\omega_0[\tilde{\tau}_{-1}\wedge \tilde{\tau}_{n}]^\lambda\Big)\lesssim_{n} (q_1+\delta)^{-n}.
\]
\end{lem}

\subsection{Null speed regime}
In this section we prove a part of (2) (when the speed is zero) in Theorem~\ref{vrjpGW-speed}.
\begin{prop}\label{vrjpGW-limsupeta}
Recall the definition of \(t^{*}\) in~\eqref{vrjpGW-eq-tstar}, if  \(q_1\E({A}^{-1/2})>1\), then \(1<t^*<\frac{3}{2}\) and
  \[\limsup_{n}\frac{\log |\eta_{n}|}{\log n}\leq t^{*}-\frac{1}{2}.\]
\end{prop}
\noindent In particular, if  \(q_1\E({A}^{-1/2})>1\), then \(\p\)-a.s., \(v(\eta)=0\); in fact,
\[
|\eta_n|=n^{(t^*-1/2)+o(1)}=o(n), \ n\rightarrow\infty.
\]
\begin{rmk}
\label{vrjpGW-rmq-limsupztovert}
  Similar arguments can be carried out for the continuous time process \((Z_{t})\), i.e.\ if  \(q_1\E({A}^{-1/2})>1\), then
  \begin{equation}\label{vrjpGW-vnulplus}
  \limsup_{t}\frac{\log |Z_{t}|}{\log t}\leq t^{*}-\frac{1}{2}.
  \end{equation}
\end{rmk}
Let us state an estimate on the tail distribution of the regeneration time \(\Gamma_1\) under \(S(\cdot)\):
\begin{lem}
\label{vrjpGW-tailGamma}
  For \(N\) large enough, in the case \(1<t^{*}<3/2\), for any \(\lambda\in(t^{*}-\frac{1}{2},1)\), there is \(\epsilon>0\) such that
   \begin{equation}
     \label{vrjpGW-eq:tailGamma}
S(\Gamma_{1}>N^{1/\lambda})\ge N^{-1+\epsilon}.
\end{equation}
\end{lem}
\noindent With the help of the above lemma, we prove Proposition~\ref{vrjpGW-limsupeta}.
\begin{proof}[Proof of Proposition~\ref{vrjpGW-limsupeta}]
  Note that \(t\mapsto \E(A^{t})\) is a convex function, and it is symmetric w.r.t.\ the line \(t=\frac{1}{2}\), where it takes the minimum,
  \begin{figure}[!h]
    \centering
    \includegraphics[width=.8\textwidth]{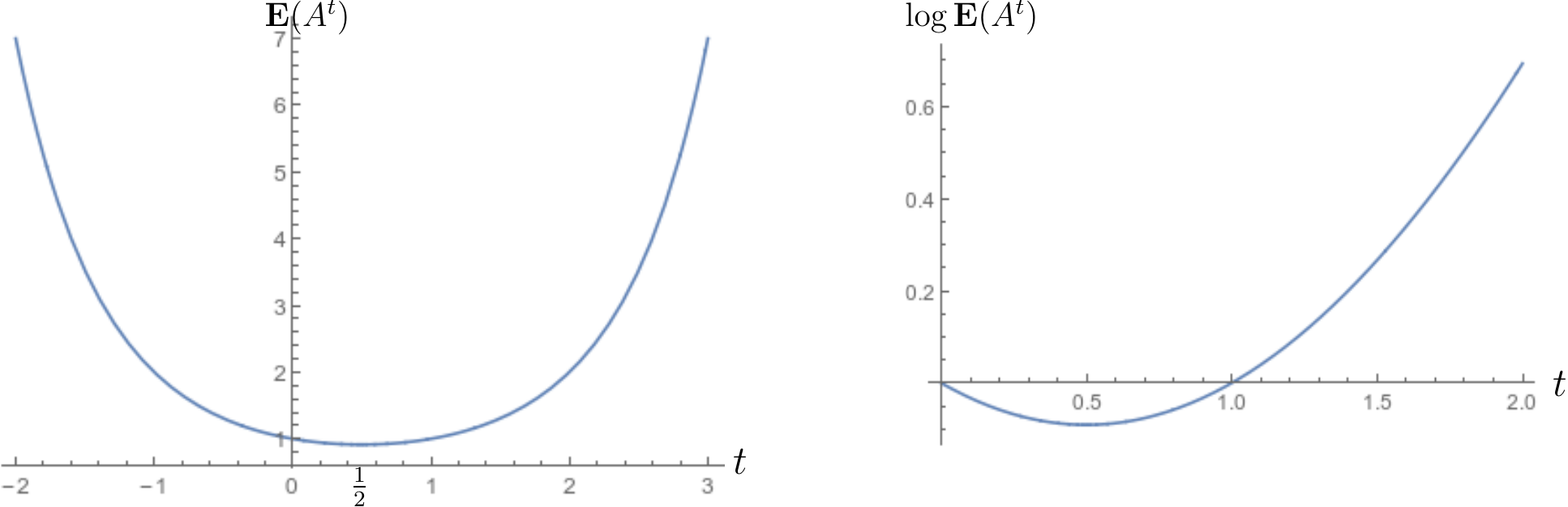}
    \caption{The function \(t\mapsto \mathbf{E}(A^t)\) and \(t\mapsto \log \mathbf{E}(A^{t})\) for \(c=1\).}
    \label{vrjpGW-fig-tmapstoEAt}
  \end{figure}
in particular \(\E(A^{-1/2})=\E(A^{3/2})\). As we have assumed that \(q_{1}\E(A^{-1/2})>1\), it follows that \(t^{*}<\frac{3}{2}\). On the other hand, since \(\E(A)=1\), obviously \(t^{*}>1\). For any \(\lambda\in (t^{*}-1/2,1)\), by Lemma~\ref{vrjpGW-tailGamma}, there exists \(\varepsilon>0\) such that
  \begin{align*}
    \p(\max_{2\leq k\leq n}(\Gamma_{k}-\Gamma_{k-1}) \leq n^{1/\lambda})&=S(\Gamma_{1}\leq n^{1/\lambda})^{n-1}\\
     &\leq (1-n^{-1+\varepsilon})^{n-1} \lesssim_{n} \exp(-n^{\varepsilon}).
  \end{align*}
Therefore,
\begin{align*}
  \sum_{n\geq 2}\p(\max_{2\leq k\leq n}(\Gamma_{k}-\Gamma_{k-1})\leq n^{1/\lambda})<\infty.
\end{align*}
By Borel-Cantelli lemma, \(\p\)-a.s., for all \(n\) large enough,
\[\Gamma_{n}\geq \max_{2\leq k\leq n}(\Gamma_{k}-\Gamma_{k-1})\geq n^{1/\lambda}.\]
It follows that \(\p\)-a.s., \(\liminf_{n}\frac{\log \Gamma_{n}}{\log n}\geq \frac{1}{\lambda}\). As \(\liminf_{n}\frac{\log \tau_{n}}{\log n}\geq \liminf_{n}\frac{\log \Gamma_{n}}{\log n}\) (see (3.1) in~\cite{aidekon2008transient}), we have
\[\limsup_{n}\frac{\log |\eta_{n}|}{\log n}\leq \lambda \xrightarrow[\text{decreasing}]{}t^{*}-\frac{1}{2}<1, \p\text{-a.s.}\]
\end{proof}

It remains to prove Lemma~\ref{vrjpGW-tailGamma}. In fact, when \(q_1\) is large, it is more likely that there will be some long branch constituting vertices of degree two on the GW tree, especially starting from the root. These branches will slow down the process and entail zero velocity. The following lemma gives a comparison between the tail distribution of the regeneration time \(\Gamma_1\) and the probability that the process wanders on these branches (which is a one dimensional random walk in random environment, that is, \((\tilde{\eta}_{n})\)).

\begin{lem}
\label{vrjpGW-compare_to_Z}
For any \(m\geq 1,\ 0<a<1\), we have
\[S(\Gamma_1>m)\geq c_5\sum_{n=1}^{\infty}q_1^n \tilde{\p}_0(\tilde{\tau}_{-1}\wedge \tilde{\tau}_n>m|A_0\in[a,\frac{1}{a}]).\]
\end{lem}
\noindent
Now we prove Lemma~\ref{vrjpGW-tailGamma} with the help of Lemma~\ref{vrjpGW-compare_to_Z} and some results on the one dimensional RWRE.

\begin{proof}[Proof of Lemma~\ref{vrjpGW-tailGamma}]
By Lemma~\ref{vrjpGW-oneMAMA}, one sees that for all \(z>0,\ 0<z_{1}<1\) and \(m\) such that \(n=\lfloor \frac{\log m}{z}\rfloor\),
\[\tilde{\p}_0(\tilde{\tau}_n\wedge \tilde{\tau}_{-1}>m|A_0\in [a,\frac{1}{a}])\gtrsim_{n} \exp(-n\left(z_1I(\frac{z}{2z_1})+(1-z_1)I(\frac{-z}{2(1-z_1)})\right))\]
where we recall that \(I(x)=\sup_{t\in \mathbb{R}}\{tx-\log \E(A^t)\}\).  For large \(m\), by Lemma~\ref{vrjpGW-compare_to_Z}, for all \(m\ge 1, \ 0<a<1\)
\begin{align*}
  S(\Gamma_1>m)&\geq c_5\max_{n:n=\lfloor \frac{\log m}{z}\rfloor}q_1^n \tilde{\p}_0(\tilde{\tau}_n\wedge \tilde{\tau}_{-1}>m|A_0\in [a,\frac{1}{a}])
\end{align*}
then by Lemma~\ref{vrjpGW-oneMAMA}, for all \(m\ge 1, \ 0<a<1, \ z>0, 0<z_{1}<1\)
\begin{align*}
S(\Gamma_{1}>m) &\gtrsim_{n} \max_{n: n=\lfloor \frac{\log m}{z}\rfloor} q_{1}^{n}\exp(-n \left( z_1I(\frac{z}{2z_1})+(1-z_1)I(\frac{-z}{2(1-z_1)})\right) )
  \\
&\gtrsim_{n} \exp\{-\frac{\log m}{z}\left(z_1I(\frac{z}{2z_1})+(1-z_1)I(\frac{-z}{2(1-z_1)})-\log q_1\right)\}.
\end{align*}
That is, for all \(z>0,\ 0< z_{1}<1\)
\[S(\Gamma_{1}>e^{nz})\gtrsim_{n} \exp\left( nz\left( \frac{\log q_{1}}{z}-\frac{z_{1}}{z}I(\frac{z}{2z_{1}})-\frac{1-z_{1}}{z}I(\frac{-z}{2(1-z_{1})}) \right) \right)\]
It follows from the proof of Lemma~\ref{vrjpGW-LDP}  that we can take \(z_{1}=\frac{1}{2}, z= (t\mapsto \log \E(A^t))'(t^*)>0\)
\[
S(\Gamma_1>e^{nz})\gtrsim_{n}  e^{nz(-t^*+1/2)}.
\]
Recall that \(1<t^{*}<\frac{3}{2}\), thus for any \(\lambda\in (t^{*}-\frac{1}{2},1) \), there exists a \(\epsilon>0\) such that (denote \(N=\lfloor  e^{\lambda nz} \rfloor\))
\[S(\Gamma_{1}>N^{1/\lambda})\ge N^{-1+\epsilon}.\]
\end{proof}
\noindent It remains to prove the comparison Lemma~\ref{vrjpGW-compare_to_Z}. We define, for \(x\neq \parent{\rho}\),
\[\tau_x=\inf\{n\geq 0;\ \eta_n=x\},\ \tau_x^*=\inf\{n>0;\ \eta_n=x\},\ \beta(x)=P_x^{w,T}(T_{\parent{x}}=\infty)\]
Note that for any \(x\in T\), \(\beta(x)\) depends only on the sub-tree \(T_x\) rooted at \(x\) and the environment \(\{A_y(\omega);y\in T_x\}\),  let us denote \(\beta\) a generic r.v.\ distributed as \(\beta(\rho)\), by transient assumption, \(\beta>0\) a.s.\ and \(\e(\beta)>0\).

Moreover, by  Markov property,
\begin{align*}
  \beta(x)&=\sum_{y:\parent{y}=x}p(x,y)[P_y^{\omega,T}(\tau_x=\infty)+P_y^{\omega,T}(\tau_x<\infty)\beta(x)]\\
&=\sum_{y:\parent{y}=x}p(x,y)[\beta(y)+(1-\beta(y))\beta(x)].
\end{align*}
Note that \(\beta(x)>0\), \(\p\)-a.s.\ hence,
\begin{equation}\label{vrjpGW-beta} \frac{1}{\beta(x)}=1+\frac{1}{A_x\sum_{y:\parent{y}=x}A_y\beta(y)}.\end{equation}
In particular, \(\beta(x)\) is increasing as a function of \(A_x\).

\begin{proof}[Proof of Lemma~\ref{vrjpGW-compare_to_Z}]
  For any vertex \(x\), let \(h(x)\) be the first descendant of \(x\) such that \(d(h(x))\geq 3\).
Let \(k_0=\inf\{k\geq 2:\ q_k>0\}\). According to the definition of \(\Gamma_1\), one observes that when \(\eta_1\neq \parent{\rho}\), 
  \[
  \Gamma_1\geq \tau_{\rho}^*\wedge \tau_{h(\eta_1)}.
  \]
 In fact, we are going to  consider the following events
\begin{align*}
&E_0=\{d(\rho)= k_0+1,\ A_{\rho}\geq a,\ A_{\rho_i}\in[a,\frac{1}{a}],\forall 1\leq i\leq k_0\}\text{ where \(\rho_i\) are children of \(\rho\)},\\
&E_1=E_0\cap \{\eta_1\neq \parent{\rho},m<\tau_{\rho}^*<\tau_{h(\eta_1)},\eta_{\tau_{\rho}^*+1}\notin\{\parent{\rho},\eta_1\}\}\cap \{\eta_n\neq \rho;\forall n\geq \tau_{\rho}^*+1\},\\
&E_2=E_0\cap \{\eta_1\neq \parent{\rho},m<\tau_{h(\eta_1)}<\tau_{\rho}^*\}\cap\{\eta_n\neq \overleftarrow{h(\eta_1)},\forall n\geq \tau_{h(\eta_1)}+1\}.
\end{align*}
As \(\Gamma_1\geq \tau_{\rho}^*\wedge \tau_{h(\eta_1)}\), we have \(E_1\cup E_2\subset E_0\cap\{D(\rho)=\infty,\Gamma_1>m\}\) and \(E_1\cap E_2=\emptyset\). So,
\[\p(E_0\cap\{D(\rho)=\infty,\Gamma_1>m\})\geq \p(E_1)+\p(E_2).\]
For \(E_1\), by strong Markov property at \(\tau_\rho^*\) and weak Markov property at time \(1\),
\begin{align*}
  P_{\rho}^{\omega,T}(E_1)&=\mathds{1}_{E_0}P_{\rho}^{\omega,T}(\{\eta_1\neq \parent{\rho},m<\tau_{\rho}^*<\tau_{h(\eta_1)},\eta_{\tau_{\rho}^*+1}\notin\{\parent{\rho},\eta_1\}\}\cap \{\eta_n\neq \rho;\forall n\geq \tau_{\rho}^*+1\})\\
&=\mathds{1}_{E_0}\sum_{i=1}^{k_0}p(\rho,\rho_i)P_{\rho_i}^{\omega,T}(m-1<\tau_{\rho}<\tau_{h(\rho_i)})\sum_{j\neq i}p(\rho,\rho_j)\beta(\rho_j).
\end{align*}
Given \(E_0\), \(p(\rho,\rho_i)\geq \frac{a^2}{k_0+1}=:c_6\) for any $1\leq i\leq k_0$. So,
\begin{align*}
P_{\rho}^{\omega,T}(E_1)\geq& \mathds{1}_{E_0}\sum_{i=1}^{k_0}c_6P_{\rho_i}^{\omega,T}(m-1<\tau_{\rho}<\tau_{h(\rho_i)})\sum_{j\neq i}c_6\beta(\rho_j)\\
=&c_6^2 \mathds{1}_{E_0}\sum_{i=1}^{k_0}\sum_{j\neq i} P_{\rho_i}^{\omega,T}(m-1<\tau_{\rho}<\tau_{h(\rho_i)})\beta(\rho_j),
\end{align*}
where $P_{\rho_i}^{\omega,T}(m-1<\tau_{\rho}<\tau_{h(\rho_i)})$ is a function of $A_{y},\  \rho_i\leq y\leq h(\rho_i)$ which is independent of $\beta(\rho_j)=\frac{A_{\rho_j}\sum_{\parent{z}=\rho_j}A_z\beta(z)}{1+A_{\rho_j}\sum_{\parent{z}=\rho_j}A_z\beta(z)}$ for $i\neq j$. Recall that given $E_0$,  we have $\P(A_\rho\in[a,a^{-1}]):=c_7>0$.
Now for any couple $(i,j)$, one has
\begin{align*}
&\e[\mathds{1}_{E_0}P_{\rho_i}^{\omega,T}(m-1<\tau_{\rho}<\tau_{h(\rho_i)})\beta(\rho_j)]\\
\geq&c_7^{k_0-1}q_{k_0}\e[\mathds{1}_{ A_{\rho_i}\in[a,a^{-1}])}P_{\rho_i}^{\omega,T}(m-1<\tau_{\rho}<\tau_{h(\rho_i)})]\e[\beta(\rho_j)\mathds{1}_{A_{\rho_j}\in[a,a^{-1}]}]\\
\geq & c_7^{-1} \e[\mathds{1}_{E_0}P_{\rho_i}^{\omega,T}(m-1<\tau_{\rho}<\tau_{h(\rho_i)})]\e[\beta(\rho)\mathds{1}_{A_{\rho}\in[a,a^{-1}]}].
\end{align*}
Note that $\E[\beta(\rho)]>0$. So we could choose $a<1$ such that $\e[\beta(\rho)\mathds{1}_{A_{\rho}\in[a,a^{-1}]}]>0$. Taking sum over $i,j$ shows that
\begin{equation}\label{vrjpGW-bdE1}
\p_\rho(E_1)\geq \frac{c_6^2}{c_7} (k_0-1)\e\Big(\mathds{1}_{E_0}\sum_{i=1}^{k_0}P_{\rho_i}^{\omega,T}(m-1<\tau_{\rho}<\tau_{h(\rho_i)})\Big).
\end{equation}
Similarly for \(E_2\), by Markov property,
\begin{align*}
  P_{\rho}^{\omega,T}(E_2)&=\mathds{1}_{E_0}P_{\rho}^{\omega,T}(\{\eta_1\neq \parent{\rho},m<\tau_{h(\eta_1)}<\tau_{\rho}^*\}\cap \{\eta_n\neq \overleftarrow{h(\eta_1)};\forall n\geq \tau_{h(\eta_1)}+1\})\\
&=\mathds{1}_{E_0}\sum_{i=1}^{k_0}p(\rho,\rho_i)P_{\rho_i}^{\omega,T}(m-1<\tau_{h(\rho_i)}<\tau_{\rho})\beta({h(\rho_i)})\\
&\geq c_6 \mathds{1}_{E_0}\sum_{i=1}^{k_0}P_{\rho_i}^{\omega,T}(m-1<\tau_{h(\rho_i)}<\tau_{\rho})\beta({h(\rho_i)}).
\end{align*}
To get rid of dependence between $P_{\rho_i}^{\omega,T}(m-1<\tau_{h(\rho_i)}<\tau_{\rho})$ and $\beta({h(\rho_i)})$, we note that
\[
P_{\rho_i}^{\omega,T}(m-1<\tau_{h(\rho_i)}<\tau_{\rho})\geq P_{\rho_i}^{\omega,T}(m-1<\tau_{\parent{h(\rho_i)}}<\tau_{\rho}, \eta_{\tau_{\parent{h(\rho_i)}}+1}=h(\rho_i))
\]
which by Markov property is $P_{\rho_i}^{\omega,T}(m-1<\tau_{\parent{h(\rho_i)}}<\tau_{\rho})\frac{A_{h(\rho_i)}}{A^{-1}_{\parent{h(\rho_i)}}+A_{h(\rho_i)}}$. This term and \(\beta({h(\rho_i)})\) are both increasing on \(A_{h(\rho_i)}\). FKG inequality conditionally on $\{T; A_u, u\neq h(\rho_i)\}$ entails
\begin{align}\label{vrjpGW-bdE2}
  \p(E_2)&\geq c_6\e\Big(\mathds{1}_{E_0}\sum_{i=1}^{k_0}P_{\rho_i}^{\omega,T}(m-1<\tau_{\parent{h(\rho_i)}}<\tau_{\rho})\frac{A_{h(\rho_i)}}{A^{-1}_{\parent{h(\rho_i)}}+A_{h(\rho_i)}}\Big)\times\e(\beta(\rho))\nonumber\\
&=c_8\e[\mathds{1}_{E_0}\sum_{i=1}^{k_0}P_{\rho_i}^{\omega,T}(m-1<\tau_{\parent{h(\rho_i)}}<\tau_{\rho})\mathds{1}_{A_{\parent{h(\rho_i)}}\geq a}]\E[\frac{aA_\rho}{1+aA_\rho}],
\end{align}
with  \(c_8:=c_6\e(\beta(\rho))>0\). Clearly, $\E[\frac{aA_\rho}{1+aA_\rho}]>0$ for any $a>0$.
Combining~\eqref{vrjpGW-bdE1} with~\eqref{vrjpGW-bdE2} yields that
\begin{align}\label{vrjpGW-bdS}
  \p(E_1)+\p(E_2)&\geq c_9\e\Big(\mathds{1}_{E_0}\sum_{i=1}^{k_0}P_{\rho_i}^{\omega,T}(\tau_{\rho}\wedge \tau_{\parent{h(\rho_i)}}>m-1)\mathds{1}_{A_{\parent{h(\rho_i)}}\geq a}\Big)\nonumber\\
&\geq c_9k_{0}\Q(E_0)\p\Big(\tau_{\parent{\rho}}\wedge \tau_{\parent{h(\rho)}}>m-1, A_{\parent{h(\rho)}}\geq a|A_{\rho}\in[a,\frac{1}{a}]\Big)\nonumber\\
&\geq c_{10}\p\Big(\tau_{\parent{\rho}}\wedge \tau_{\grandpa{h(\rho)}}>m-1; |h(\rho)|\geq2|A_{\rho}\in[a,\frac{1}{a}]\Big).
\end{align}
where $\grandpa{h(\rho)}$ is the grand parent of $h(\rho)$. Let us go back to \(S(\Gamma_1>m)\). As \(\p(d(\rho)\geq 3,D(\rho)=\infty)>0\), recall that 
\begin{align*}
  S(\Gamma_1>m)&=\p(\Gamma_1>m|d(\rho)\geq 3,D(\rho)=\infty)\\
&\geq \p(E_0\cap \{D(\rho)=\infty,\Gamma_1>m\})\\
&\geq \p(E_1)+\p(E_2).
\end{align*}
By~\eqref{vrjpGW-bdS}, taking \(c_{5}= c_{10}\), we have
\begin{align*}
S(\Gamma_1>m)&\geq c_{5}\p\Big(\tau_{\parent{\rho}}\wedge \tau_{\grandpa{h(\rho)}}>m-1, |h(\rho)|\geq 2|A_{\rho}\in[a,\frac{1}{a}]\Big)\\
&= c_{5}q_1^2\sum_{n=1}^{\infty}q_1^n \tilde{\p}_0(\tilde{\tau}_{-1}\wedge \tilde{\tau}_n>m-1|A_0\in[a,\frac{1}{a}]).
\end{align*}
\end{proof}

\subsection{Positive speed on big tree and asymptotic of \(|Z_{t}|\) on small tree}
This subsection is devoted to the proof of the following propositions, firstly when the tree is big (i.e.\ \(q_{1}\) small), the RWRE has positive speed; when the tree is small (\(q_{1}\) large), we can compute exactly the asymptotic behavior of  \(|\eta_{n}|\) and \(|Z_{t}|\).
\begin{prop}
\label{vrjpGW-prop-v-positive}
 If \(q_1\E({A}^{-1/2})<1\), then
\begin{equation}\label{vrjpGW-vpos}
v(\eta)>0\textrm{ and } v(Z)>0.
\end{equation}
As a consequence, also \(v(Y)>0\).
\end{prop}
\begin{prop}
\label{vrjpGW-prop-logZ-logt}
 Assume that \(q_1\E({A}^{-1/2})>1\), we have \(\p\)-a.s.
\begin{equation}
\label{vrjpGW-eq-Z_tOverlogt}
\lim_{n\rightarrow\infty}\frac{\log |\eta_n|}{\log n}=\lim_{t\rightarrow\infty}\frac{\log |Z_t|}{\log t}= t^*-1/2\in (1/2, 1)
\end{equation}
where \(t^{*}=\sup\{t\in \mathbb{R},\ \E(A^{t})q_{1}\leq 1\}\).
\end{prop}
\noindent Let us give some definitions and heuristics before proving these propositions, write, for \(n\geq0\),
\begin{align*}
  &\tau_{n}(\eta)=\inf\{k\geq 0;\ |\eta_{k}|=n\}\textrm{ and }\tau_{n}(Z)=\inf\{t\geq 0;\ |Z_{t}|=n\}
\end{align*}
the hitting times of the \(n\)-th generation for \(\eta\) and \(Z\) respectively. As a consequence of the law of large numbers, \(\p\)-a.s.,
\[\lim_{n\to \infty}\frac{\tau_{n}(\eta)}{n}=\frac{1}{v(\eta)}\ \text{and }\ \lim_{n\to \infty}\frac{\tau_{n}(Z)}{n}=\frac{1}{v(Z)}.\]
The study of the speed is reduced to the study of \(\tau_n(\eta)\) and \(\tau_n(Z)\). For any \(x\in T\), \(n\geq-1\), let \(N_x\) and \(N_n\) denote the time spent by the walk \(\eta\) at \(x\) and at the \(n\)-th generation respectively:
\[
N(x)=\sum_{k\geq 0}\mathds{1}_{\eta_{k}=x},\ \ N_{n}=\sum_{|x|=n}N(x),
\]
observe that
\[
\tau_n(\eta)\leq \sum_{k=-1}^n N_k,\ \ E^{\omega, T}[\tau_n(Z)\vert \eta] \leq \sum_{x:-1\leq |x|\leq n}N_x\frac{A_x}{1+A_xB_x},
\]
where \(B_x:=\sum_{y:\parent{y}=x}A_y\).

In what follows, we actually study \(N_n\) for large \(n\) to show that \(\liminf_{n}\frac{\sum_{k=-1}^n N_k}{n}<\infty\), \(\p\)-a.s.\ The heuristics is the following. Fix some \(n_{0},K_{0}\) (to choose later), pick some vertex \(y\) at the \(n\)-th generation, if \(y\) roughly lies in a subtree of height \(n_{0}\) with more than \(K_{0}\) leaves, then the random walk will immediately go down, thus \(\e(N_{y})\) will be small c.f.\ Figure~\ref{vrjpGW-fig-yhatyycheck} left. Otherwise, we seek a down going path \(\hat{y},\ldots,y,\ldots,\check{y}\) such that every vertex in this path does not branch much except for the two ends, and we need these two ends have more than \(K_{0}\) descendants after \(n_{0}\) generations. In such configuration, we can compare the random walk to the one dimensional one, and once the walker reaches one of the ends, it immediately leaves our path \(\hat{y},\ldots,\check{y}\) c.f.\ Figure~\ref{vrjpGW-fig-yhatyycheck} right.

\begin{figure}[!h]
  \centering
  \includegraphics[width=.6\textwidth]{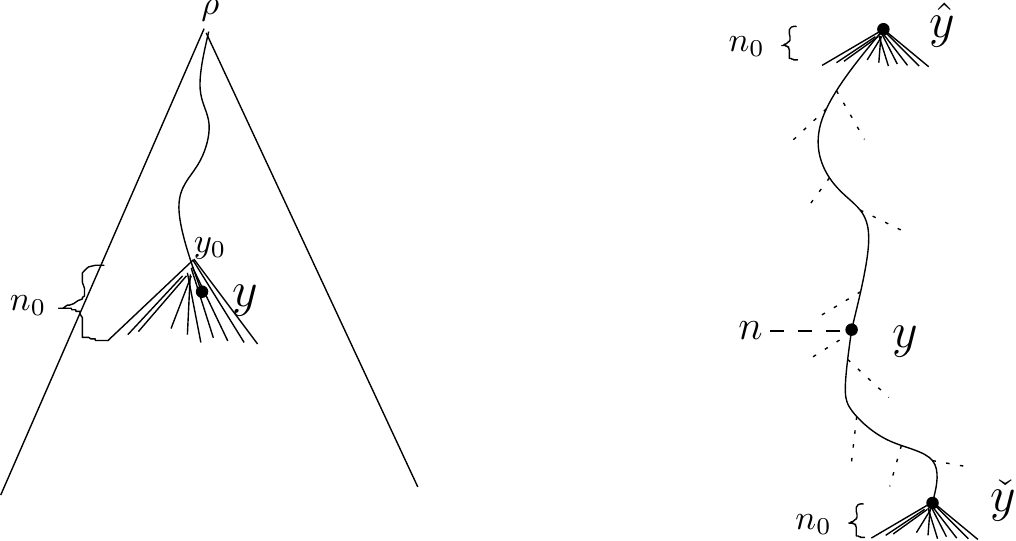}
  \caption{Two cases to bound \(\e(N_{y})\).}
  \label{vrjpGW-fig-yhatyycheck}
\end{figure}
If the root has more than \(K_{0}\) descendants after \(n_{0}\) generations, then we can always find \(\hat{y}\). Otherwise, we need to take \(n\) large and use the Galton Watson structure. To handle this issue, let us introduce the following notations. For the GW tree \(T\), let \(Z^T_n\) be the number of vertices at the \(n\)-th generation. By Lemma 4.1 of~\cite{aidekon2008transient}, we have for any \(K_{0}\geq 1\),
\[\e_{\text{GW}}(Z^T_{n}\mathds{1}_{Z^T_{n}\leq K_{0}})\leq K_{0}n^{K_{0}}q_{1}^{n-K_{0}}.\]
Let \(r\in (q_{1},1)\) be some real number to be chosen later, let
\[n_{0}=n_{0}(K_{0},r):=\inf \{n\geq 1,\ \e_{\text{GW}}(Z^T_{n}\mathds{1}_{Z^T_{n}\leq K_{0}})\leq r^{n}\},\]
which is thus a finite integer. In fact, \(K_{0}\) will be chosen according to Corollary~\ref{vrjpGW-coro-1}. Define (recall that  we write \(u<v\) if
\(u\) is an ancestor of \(v\).)
\begin{align*}
   Z^T(u,n)=|\{x\in T;\ u<x,|x|=|u|+n\}|.
\end{align*}
Let \(T_{n_{0}}\) be a tree induced from \(T\) in the following way: starting from the root \(\rho\), \(y\) is a child of \(x\) in \(T_{n_{0}}\) if \(x<y\) and \(|y|=|x|+n_{0}\). Define a subtree \(\mathcal{W}\) of \(T_{n_{0}}\) by
\[\mathcal{W}=\{x\in T_{n_{0}}:\ \forall u\in T_{n_{0}}, u<x \Rightarrow Z^T(u,n_{0})\leq K_{0}\}.\]
\begin{figure}[!h]
  \centering
  \includegraphics[width=.7\textwidth]{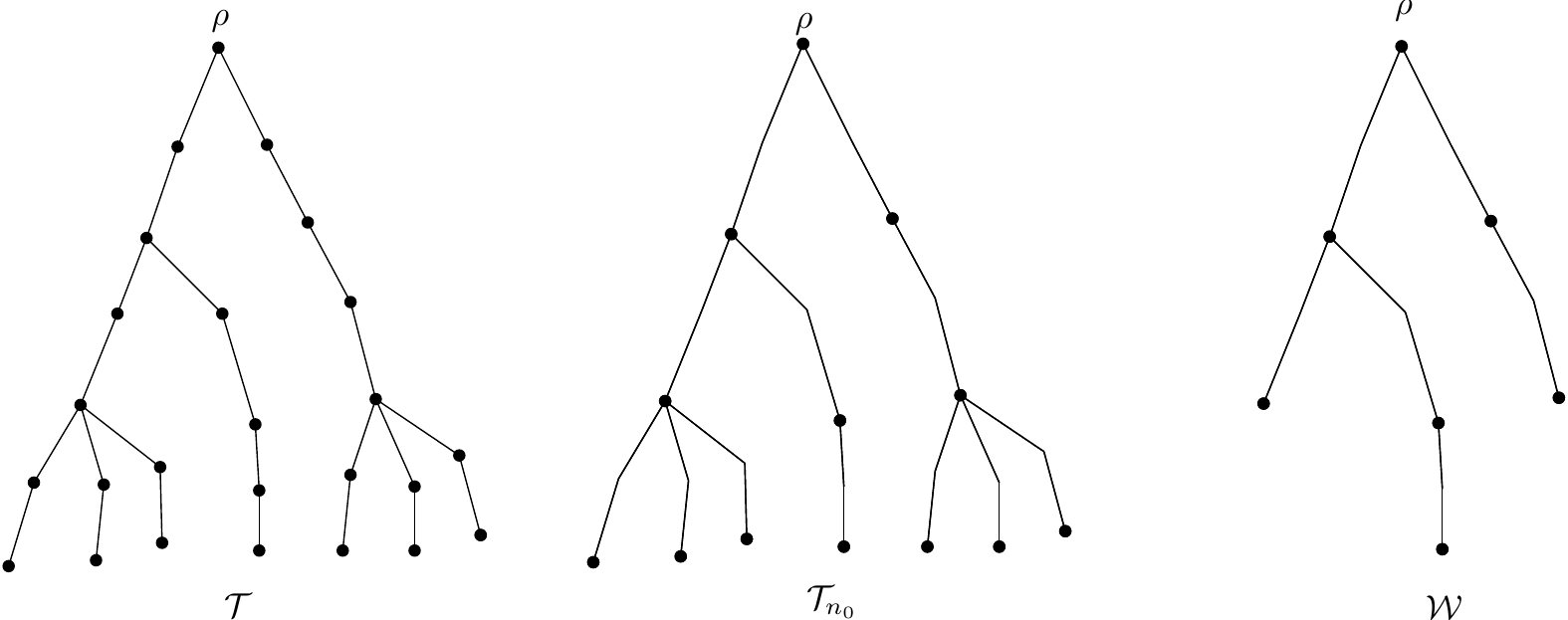}
  \caption{An example in the case \(K_0=n_0=2\).}
  \label{vrjpGW-fig:T_Tn0_W}
\end{figure}

Let \(W_{k}\) be the population of the \(k\)-th generation of \(\mathcal{W}\), \(\mathcal{W}\) is a sub critical Galton Watson tree of mean offspring \(\e_{\text{GW}}(Z^T_{n_0}\mathds{1}_{Z^T_{n_0}\leq K_{0}})\leq r^{n_{0}}\); in particular, for any \(k\geq 0\), \(\e_{\text{GW}}(W_{k})\leq r^{kn_{0}}\).

For any \(y\in T\), let \(y_{0}\) be the youngest ancestor of \(y\) in \(T_{n_{0}}\). For \(n\geq n_{0}\), let \(j=\lfloor \frac{n}{n_{0}}\rfloor\geq 1\) so that \(jn_{0}\leq n< (j+1)n_{0}\). Define
\begin{equation}
  \label{vrjpGW-defi_Nn1}
  N_{n,1}=\sum_{|y|=n}N(y)\mathds{1}_{Z^T(y_{0},n_{0})>K_{0}},\ \
N^*_{n,1}=\sum_{|y|=n}N(y)\frac{A_y}{1+A_yB_y}\mathds{1}_{Z^T(y_{0},n_{0})>K_{0}}
\end{equation}
\begin{equation}
  \label{vrjpGW-defi_Nn2}
  N_{n,2}=\sum_{|y|=n}N(y)\mathds{1}_{Z^T(y_{0},n_{0})\leq K_{0},\ y_{0}\notin \mathcal{W}},\ \ N^*_{n,2}=\sum_{|y|=n}N(y)\frac{A_y}{1+A_yB_y}\mathds{1}_{Z^T(y_{0},n_{0})\leq K_{0},\ y_{0}\notin \mathcal{W}}
\end{equation}
\begin{lem}
  \label{vrjpGW-boundwithouW}
There exist \(r\in (q_{1},1)\) and \(K_{0}>0\), such that, with the definitions of \(n_{0},N_{n,1},N_{n,1}^{*}\) above, for some constant \(L>0\), for any \(n\geq n_{0}\)
\begin{equation}
  \label{vrjpGW-eq_estimate_Nn}
  \e(N_{n,1})\leq L,\ \ \   \e(N^*_{n,1}) \leq L.
\end{equation}
\end{lem}
\begin{lem}
\label{vrjpGW-boundwithouW2}
With the same assumption as in Lemma~\ref{vrjpGW-boundwithouW}, let \(t^{*}\) be defined as in~\eqref{vrjpGW-eq-tstar}, then
\begin{equation}
  \label{vrjpGW-eq_estimate_Nnstar}
\e(N^\lambda_{n,2})\leq L, \ \ \ \e((N^*_{n,2})^\lambda)\leq L
\end{equation}
holds in either of the following cases
\begin{enumerate}[1)]
\item \(\lambda\in [0,1]\) and \(t^{*}>3/2\)
\item \(\lambda\in [0,t^{*}-\frac{1}{2})\) and \(t^{*}\leq 3/2\).
\end{enumerate}
\end{lem}
We are prepared to prove Proposition~\ref{vrjpGW-prop-v-positive} and Proposition~\ref{vrjpGW-prop-logZ-logt}.
\begin{proof}[Proof of Proposition~\ref{vrjpGW-prop-v-positive}]
Since \(q_1\E({A}^{-1/2})<1\), \(t^*>3/2\). We choose \(\lambda=1\).  As \(\mathcal{W}\) is finite a.s., if \(\chi=(\operatorname{height}(\mathcal{W})+1)n_{0}\) (where for a finite tree \(T\), \(\operatorname{height}(T):=\max_{x\in T}|x|\)), then
\[
  \text{for all }n\geq \chi,\ N_{n}\leq N_{n,1}+N_{n,2}.
\]
By Lemma~\ref{vrjpGW-boundwithouW},~\ref{vrjpGW-boundwithouW2}, for any \(n\geq n_{0}\),
\[\e(N_{n},n\geq \chi)\leq 2L.\]
Thus,
\[
\liminf_{n\rightarrow\infty}\e\Big[\frac{\sum_{i=\chi}^n N_n}{n}\Big]\leq 2L.
\]
By Fatou's lemma, a.s.
\[\liminf_{n\to\infty}\frac{\sum_{k=-1}^nN_k}{n}=\liminf_{n\to \infty} \frac{\sum_{k=\chi}^{n}N_{k}}{n}<\infty.\]
Therefore,
\begin{align*}
 \frac{1}{v(\eta)}=\liminf_{n\to \infty} \frac{\tau_{n}}{n} \leq \liminf_{n\to \infty}\frac{\sum_{k=-1}^{n}N_{k}}{n}<\infty.
\end{align*}
This implies that \(v(\eta)>0\).

The case for \(Z_{t}\) can be treated in a similar manner with \(N^*_{n}\) instead of \(N_{n}\). Finally, to prove \(v(Y)>0\), it is enough to recall \(Z_{D(t)}=Y_t\) where \(D(t)=\sum_x (L_x(t)^2-c^{2})\) and note that
\[\frac{D(t)}{t}=\frac{\sum_x (L_x(t)+c)(L_{x}(t)-c)}{\sum_x (L_x(t)-c)}\geq 2c>0.\]
It follows that
\[v(Y)=\lim_{t\to\infty} \frac{|Y_t|}{t}=\lim_{t\to\infty} \frac{|Z_{D(t)}|}{t}\geq v(Z)\liminf_{t\to\infty} \frac{D(t)}{t}\geq 2c v(Z).\]
\end{proof}

\begin{proof}[Proof of Proposition~\ref{vrjpGW-prop-logZ-logt}]
If \(q_1\E({A}^{-1/2})\geq 1\), \(\lambda<t^*-1/2\leq 1\). Let \(N_i(Z)\) be the time spent at the \(i\)-th generation by \((Z_t)\). Let \(\Gamma_k(Z)\) be the regenerative times corresponding to \((Z_t)_{t\geq0}\). Let \(u(n)\) be the unique integer such that \(\Gamma_{u(n)}\leq \tau_n(Z)<\Gamma_{u(n)+1}\). Then,
\begin{align*}
\frac{\Gamma_{u(n)}(Z)^\lambda}{n}\leq \frac{\sum_{k\leq u(n)}(\Gamma_{k}(Z)-\Gamma_{k-1}(Z))^\lambda}{n}&=\frac{\sum_{k\leq u(n)}(\sum_{i=|Z_{\Gamma_{k-1}(Z)}|}^{i=|Z_{\Gamma_{k}(Z)}|-1}N_i(Z))^\lambda}{n}\\
&\leq \frac{\sum_{i\leq n}N_i(Z)^\lambda}{n}.
\end{align*}
Taking limit yields that
\begin{align*}
\liminf_{n\rightarrow\infty}\frac{\Gamma_{u(n)}(Z)^\lambda}{n}\leq \liminf_{n\rightarrow\infty}\frac{\sum_{k\leq u(n)}(\Gamma_{k}(Z)-\Gamma_{k-1}(Z))^\lambda}{n}\leq \liminf_{n\rightarrow\infty}\frac{\sum_{i=\chi}^n N_i(Z)^\lambda}{n}.
\end{align*}
Applying Jensen's inequality then Lemma~\ref{vrjpGW-boundwithouW} and Lemma~\ref{vrjpGW-boundwithouW2} implies that
\begin{align*}
\e[N_n(Z)^\lambda; n\geq \chi]\leq \e[ \e^{\omega,T}[N_n(Z); n\geq \chi\vert \eta]^\lambda]\leq \e[(N_{n,1}^*+N^{*}_{n,2})^\lambda, n\geq\chi]\leq 2L.
\end{align*}
It follows from Fatou's lemma that
\[
\liminf_{n\rightarrow\infty}\frac{\Gamma_{u(n)}(Z)^\lambda}{n}\leq \liminf_{n\rightarrow\infty}\frac{\sum_{k\leq u(n)}(\Gamma_{k}(Z)-\Gamma_{k-1}(Z))^\lambda}{n} <\infty.
\]
By law of large numbers, 
\[
\lim_{n\rightarrow\infty} \frac{n}{u(n)}=\e_S[|Z_{\Gamma_1(Z)}|]<\infty,\textrm{ and }\lim_{n\rightarrow\infty} \frac{\sum_{k\leq n}(\Gamma_{k}(Z)-\Gamma_{k-1}(Z))^\lambda}{n}=\e_S[\Gamma_1(Z)^\lambda].
\]
Therefore there exists a constant \(C\in(0,\infty)\) such that 
\[
\liminf_{n\rightarrow\infty}\frac{\Gamma_{n}(Z)^\lambda}{n}< C.
\]
Note that \(|Z_t|\geq \#\{k:\Gamma_k(Z)<t\}\). So we get \(|Z_t|\geq t^\lambda/C\) for all sufficiently large \(t\). We hence deduce that
\[
\liminf_{t\rightarrow\infty}\frac{\log |Z_t|}{\log t}\geq \lambda.
\]
Letting \(\lambda\uparrow t^*-1/2\) yields
\begin{equation}\label{vrjpGW-vnul}
\liminf_{n\rightarrow\infty}\frac{\log |Z_t|}{\log t}\geq t^*-1/2.
\end{equation}
The result follows from Remark~\ref{vrjpGW-rmq-limsupztovert}. Similar arguments can be applied to \(\lim_{n\to\infty}\frac{\log|\eta_{n}|}{\log n}\).
\end{proof}
It remains to show the main Lemmas~\ref{vrjpGW-boundwithouW},\ref{vrjpGW-boundwithouW2}. Let us first state some preliminary results. As the walk is transient, the support of the random walk should be slim. This is formulated in the following lemma:
\begin{lem}
\label{vrjpGW-bdT}
  There exists a constant \(c_{11}>0\) such that for any \(n\geq 1\), \(\displaystyle \e(\sum_{|x|=n}\mathds{1}_{\tau_x<\infty})\leq c_{11}\).
\end{lem}
\noindent
The following lemma shows that, the escape probability is relatively large. In fact, we cannot show that \(\e(\frac{1}{\beta(\rho)})<\infty\) for all \(q_{1}>0\), however as the GW tree branches anyway, there will be a large copies of independent sub-trees, we show \(\e(\frac{1}{\sum_{i=1}^{K}\beta_{i}})<\infty\) instead.
\begin{lem}
\label{vrjpGW-lem:beta}
Consider i.i.d.\ copies of GW trees \(T^{(i)}\) rooted at \(\rho^{(i)}\) with independent environment \(\omega^{(i)}\), for each \(T^{(i)}\), define \(\beta_{i}=P^{\omega^{(i)},T^{(i)}}_{\rho^{(i)}}(\tau_{\parent{\rho^{(i)}}}=\infty)\). There exists an integer \(K=K(q_1,c)\geq 1\) such that
\[\e(\frac{1}{\sum_{i=1}^K\beta_i})\leq c_{12}<\infty\ \text{ and }\ \e(\frac{1}{\sum_{i=1}^KA_{\rho^{(i)}}\beta_{i}} )<c_{12}<\infty.\]
Moreover, if \(q_1\xi_{2}<1\), then \(\displaystyle \e(\frac{1}{\beta(\rho)})\leq c_{12}<\infty\) and \(\e(\frac{1}{A_{\rho}\beta(\rho)})<c_{12}<\infty\).
\end{lem}
\begin{rmk}
  In fact, if \(q_1\E(A^{-2})<1\), a proof similar to Proposition 2.3 of~\cite{aidekon2008transient} shows that \(\eta\) has positive speed, in particular, the VRJP on any regular tree (except \(\mathbb{Z}\)) admits positive speed.
\end{rmk}
\begin{coro}
\label{vrjpGW-coro-1}
There exists \(K_0\geq K\), such that for \(k\in\{2,4\}\)
\[\e(\frac{1}{\sum_{1}^{K_{0}}A_{\rho^{(i)}}^{k}\beta_{i}^{k}}) <c_{13}<\infty.\]
\end{coro}
The proof of Lemma~\ref{vrjpGW-bdT},~\ref{vrjpGW-lem:beta} and Corollary~\ref{vrjpGW-coro-1} will be postponed to the Appendix~\ref{vrjpGW-rwreAppendix}, let us state the consequence of these preliminary results. Recall that \(Z_n^{T}\) is the population at generation \(n\), and that for any \(x\in T\), \(\tau_x\) is the first hitting time, \(\tau_x^*\) the first return time to \(x\). For \(u,v\in T\) such that \(u< v\)  define
\begin{align*}
  p_1(u,v)=P_u^{\omega,T}(\tau_{\parent{u}}=\infty,\tau_u^*=\infty,\tau_v=\infty)
\end{align*}
\begin{lem}
  \label{vrjpGW-peu}
  For any \(n\geq 2\) and \(k\in\{1,2,4\}\), consider \(K_{0}\) as in Corollary~\ref{vrjpGW-coro-1}, we have
\[\e\Big(\mathds{1}_{Z^T_n>K_0}\sum_{|u|=n}\frac{1}{p_1(\rho,u)^k}\Big)<c_{14}^n<\infty.\]
In addition, 
\begin{equation}\label{vrjpGW-eq-condA0}
\e\Big(\mathds{1}_{Z^T_n>K_0}\sum_{|u|=n}\frac{1}{p_1(\rho,u)^k}\Big\vert A_\rho\Big)<c_{14}^n(1+\frac{1}{A_\rho}).
\end{equation}
\end{lem}
\begin{proof}[Proof of Lemma~\ref{vrjpGW-peu}]
  Fix \(n\geq 2\), let \(\Upsilon_0:=\inf\{l\geq 1;\ Z_l>K_0\}\), then \(\{Z_n^{T}>K_0\}=\{\Upsilon_0\leq n\}\). For any \(u\in T\) such that \(|u|\geq \Upsilon_0\), let \(U\) be its ancestor at the \(\Upsilon_0\)-th generation. By Markov property,
\begin{equation}
\label{vrjpGW-lowerbdp1}
\begin{aligned}
    p_1(\rho,u)&\geq \sum_{|y|=\Upsilon_0-1}P_{\rho}^{\omega,T}(\tau_y<\tau_{\rho}^*)P_y^{\omega,T}(\tau_{\parent{y}}=\infty,\tau_U=\infty)\\
&\geq  \sum_{|y|=\Upsilon_0-1} \prod_{i=0}^{\Upsilon_0-2}p(y_i,y_{i+1}) P_y^{\omega,T}(\tau_{\parent{y}}=\infty,\tau_U=\infty)
  \end{aligned}
\end{equation}
where \{\(y_0(=\rho),y_1,\ldots,y_{\Upsilon_0-1}(=y)\)\} is the unique path connecting \(\rho\) and \(y\). Note that if \(\parent{U}=y\), then
\begin{align*}
  &P_y^{\omega,T}(\tau_{\parent{y}}=\infty,\tau_U=\infty)=\sum_{z:\parent{z}=y,z\neq U}p(y,z)\beta(z)+\sum_{z:\parent{z}=y,z\neq U}p(y,z)(1-\beta(z))  P_y^{\omega,T}(\tau_{\parent{y}}=\infty,\tau_U=\infty).
 \end{align*}
Otherwise
\begin{align*}
  &P_y^{\omega,T}(\tau_{\parent{y}}=\infty,\tau_U=\infty)=\sum_{z:\parent{z}=y}p(y,z)\beta(z)+\sum_{z:\parent{z}=y}p(y,z)(1-\beta(z))  P_y^{\omega,T}(\tau_{\parent{y}}=\infty,\tau_U=\infty)
 \end{align*}
It follows that in both cases,
\begin{align*}
  P_y^{\omega,T}(\tau_{\parent{y}}=\infty,\tau_U=\infty)&=\frac{\sum_{z:\parent{z}=y}\mathds{1}_{z\neq U}p(y,z)\beta(z) }{p(y,\parent{y})+p(y,U)+\sum_{z:\parent{z}=y}\mathds{1}_{z\neq U}p(y,z)\beta(z) }\\
&\geq\frac{\sum_{z:\parent{z}=y}\mathds{1}_{z\neq U}A_yA_z\beta(z) }{1+A_yA_U+\sum_{z:\parent{z}=y}\mathds{1}_{z\neq U}A_yA_z\beta(z) }\\
&\geq \frac{A_y}{1+A_y}\frac{1}{1+A_U}\frac{\sum_{z:\parent{z}=y}\mathds{1}_{z\neq U}A_z\beta(z) }{1+\sum_{z:\parent{z}=y}\mathds{1}_{z\neq U}A_z\beta(z) }
\end{align*}
Plugging it into~\eqref{vrjpGW-lowerbdp1} yields that
\begin{align*}
  p_1(\rho,u)&\geq \sum_{|y|=\Upsilon_0-1}\prod_{i=0}^{\Upsilon_0-2}p(y_i,y_{i+1}) \frac{A_y}{1+A_y}\frac{1}{1+A_U}\frac{\sum_{z:\parent{z}=y}\mathds{1}_{z\neq U}A_z\beta(z) }{1+\sum_{z:\parent{z}=y}\mathds{1}_{z\neq U}A_z\beta(z)}\\
&\geq \frac{1}{1+A_U} \min_{|y|=\Upsilon_0-1}\left(\prod_{i=0}^{\Upsilon_0-2}p(y_i,y_{i+1}) \frac{A_y}{1+A_y}\right)\cdot \frac{\sum_{z:|z|=\Upsilon_0,z\neq U}A_z\beta(z)}{1+\sum_{z:|z|=\Upsilon_0,z\neq U}A_z\beta(z) }
\end{align*}
Thus, for \(k\in\{1,2,4\}\),
\begin{equation*}
\frac{1}{p_1(\rho,u)^k}\leq (1+A_U)^k\frac{1}{\min_{|y|=\Upsilon_0-1}\left(\prod_{i=0}^{\Upsilon_0-2}p(y_i,y_{i+1})\frac{A_y}{1+A_y}\right)^k}\Big (1+\frac{1}{\sum_{z:|z|=\Upsilon_0,z\neq U}A_z\beta(z)}\Big)^k.
\end{equation*}
Given the tree \(T\), by integrating w.r.t. \(\P(d\omega)\), we have
\begin{align*}
\mathds{1}_{n\geq \Upsilon_0}\sum_{|u|=n}\e^T\Big(\frac{1}{p_1(\rho,u)^k}\Big)&\leq \e^T \left(\frac{1}{\min_{|y|=\Upsilon_0-1}\left(\prod_{i=0}^{\Upsilon_0-2}p(y_i,y_{i+1})\frac{A_y}{1+A_y}\right)^k }\right)\\
&\times \sum_{|U|=\Upsilon_0}Z^T(U,n-\Upsilon_0)\e^T[(1+A_U)^k]\e^T\left(\Big(1+\frac{1}{\sum_{z:|z|=\Upsilon_0,z\neq U}A_z\beta(z)}\Big)^k\right)
\end{align*}
It follows from Lemma~\ref{vrjpGW-lem:beta} for \(k=1\) or Corollary~\ref{vrjpGW-coro-1} for \(k=2,4\) that
\begin{align*}
  &\E_{\Q} \left( \mathds{1}_{n\geq \Upsilon_0}\sum_{|u|=n}\frac{1}{p_1(\rho,u)^k} \biggr\vert \Upsilon_0, Z_l;0\leq l\leq \Upsilon_0 \right)\\
\leq& c_{15}\mathds{1}_{n\geq \Upsilon_0}\e^T \left(\frac{1}{\min_{|y|=\Upsilon_0-1}\left(\prod_{i=0}^{\Upsilon_0-2}p(y_i,y_{i+1})\frac{A_y}{1+A_y}\right)^k }\right)\times \sum_{|U|=\Upsilon_0}\E[(1+A)^k]b^{n-\Upsilon_0}\\
\leq&c_{16}\mathds{1}_{n\geq \Upsilon_0}\sum_{|y|=\Upsilon_0-1}\e^T\left[ \left(\prod_{i=0}^{\Upsilon_0-2}\frac{(1+A_{y_i})(1+B_{y_i})}{A_{y_i}A_{y_{i+1}}}\frac{1+A_y}{A_y}\right)^k\right]\sum_{|U|=\Upsilon_0}b^{n-\Upsilon_0}.
\end{align*}
By independence of \(A_x, x\in T\), we see that 
\[
\e^T \left[\left(\prod_{i=0}^{\Upsilon_0-2}\frac{(1+A_{y_i})(1+B_{y_i})}{A_{y_i}A_{y_{i+1}}}\frac{1+A_y}{A_y}\right)^k\right]\leq c_{17}^{\Upsilon_0-1},
\]
with \(c_{17}\in (1,\infty)\). Consequently,
\begin{align*}
\E_{\Q} \left( \mathds{1}_{n\geq \Upsilon_0}\sum_{|u|=n}\frac{1}{p_1(\rho,u)^k} \right)&\leq \E_{\Q}\Big(c_{16}\mathds{1}_{n\geq \Upsilon_0}\sum_{|y|=\Upsilon_0-1} c_{17}^{\Upsilon_0-1}\sum_{|U|=\Upsilon_0}b^{n-\Upsilon_0}\Big)\\
\leq&c_{16}K_0\E_{\Q}\Big(\mathds{1}_{n\geq \Upsilon_0}c_{17}^{n-1}Z^T_n\Big)\\
\leq& c_{18}(c_{17}b)^{n}<\infty.
\end{align*}
\eqref{vrjpGW-eq-condA0} follows in the same way.
\end{proof}
\begin{proof}[Proof of Lemma~\ref{vrjpGW-boundwithouW}]
We only bound \(\e(N_{n,1})\), the argument for \(\e(N_{n,1}^{*})\) is similar. For any \(y\in T\) at the \(n\)-th generation  such that  \(Z^T(y_{0},n_{0})>K_{0}\), let \(Y\) be the youngest ancestor of \(y\) such that \(Z^T(Y,n_{0})>K_{0}\). Clearly, \(y_{0}\leq Y\leq y\). So,
\[N_{n,1}=\sum_{|y|=n}N(y)\mathds{1}_{Z^T(y_{0},n_{0})>K_{0}}\leq \sum_{|y|=n}N(y)\mathds{1}_{y_{0}\leq Y\leq y}.\]
Taking expectation w.r.t.\ \(E_{\rho}^{\omega,T}\) implies that
\[E^{\omega,T}(N_{n,1})\leq \sum_{|y|=n}E^{\omega,T}(N(y))\mathds{1}_{y_{0}\leq Y\leq y}=\sum_{|y|=n}P^{\omega,T}(\tau_{y}<\infty)E_{y}^{\omega,T}(N(y))\mathds{1}_{y_{0}\leq Y\leq y}.\]
Applying the Markov property at \(\tau_{Y}\) to \(E_{y}^{\omega,T}(N(y))\), we have
\[E_{y}^{\omega,T}(N(y))=G^{\tau_{Y}}(y,y)+P_{y}^{\omega,T}(\tau_{Y}<\infty)P_{Y}^{\omega,T}(\tau_{y}<\infty)E_{y}^{\omega,T}(N(y))\]
where (write \(\{(\tau_{Y}\wedge \infty)>\tau_{y}^{*}\}=\{\tau_{y}^{*}<\infty \text{ and }\tau_{y}^{*}<\tau_{Y}\}\) for short)
\[G^{\tau_{Y}}(y,y)=E_{y}^{\omega,T}(\sum_{k=0}^{\tau_{Y}}\mathds{1}_{\eta_{k}=y})=\frac{1}{1-P_{y}^{\omega,T}((\tau_{Y}\wedge \infty)>\tau_{y}^{*})}.\]
Hence
\begin{align*}
  E_{y}^{\omega,T}(N(y))&=\frac{G^{\tau_{Y}}(y,y)}{1-P_{Y}^{\omega,T}(\tau_{y}<\infty)P_{y}^{\omega,T}(\tau_{Y}<\infty)}\\
&\leq \frac{G^{\tau_{Y}}(y,y)}{1-P_{Y}^{\omega,T}(\tau_{Y}^{*}<\infty)}=\frac{G^{\tau_{Y}}(y,y)}{P_{Y}^{\omega,T}(\tau_{Y}^{*}=\infty)}.
\end{align*}
We bound \(G^{\tau_{Y}}(y,y)\) first. As \(P_{y}^{\omega,T}((\tau_{Y}\wedge \infty)>\tau_{y}^{*})\leq\sum_{z:\parent{z}=y}p(y,z)+p(y,\parent{y})P_{\parent{y}}^{\omega,T}(\tau_{y}<(\tau_{Y}\wedge \infty)) \), 
\[
1-P_{y}^{\omega,T}((\tau_{Y}\wedge \infty) >\tau_{y}^{*})\geq p(y,\parent{y})\Big(1-P_{\parent{y}}^{\omega,T}(\tau_{y}<\tau_{Y})\Big).
\]
By Lemma 4.4 of~\cite{aidekon2008transient} and~\eqref{vrjpGW-eq-solDirich-1d}, the right hand side of the above inequality is larger than
\begin{align*}
p(y,\parent{y})\Big(1-\tilde{P}_{\parent{y}}^{\omega,T}(\tilde{\tau}_{y}< \tilde{\tau}_{Y})\Big)=\frac{1}{1+A_{y}B_{y}}\frac{1}{1+A_{y}\sum_{Y<z<y}A_{z}\prod_{z<u<y}A_{u}^{2}}.
\end{align*}
where we identify \(\tilde{P}_{\parent{y}}^{\omega,T}\) to the probability of \((\tilde{\eta}_{n})\) on the segment \([\![Y,y]\!]\). Therefore,
\[
G^{\tau_{Y}}(y,y)\leq \Big(1+A_y\sum_{Y<z< y}A_z\prod_{z<u<y}A_u^2\Big)(1+A_yB_y)=:V_{y, Y}.
\]
Consequently,
\[E^{\omega,T}(N(y))\mathds{1}_{Z^T(y_{0},n_{0})>K_{0}}\leq P^{\omega,T}(\tau_{Y}<\infty) \frac{V_{y,Y}}{P_{Y}^{\omega,T}(\tau_{Y}^{*}=\infty)}\mathds{1}_{Z^T(Y,n_{0})>K_{0},\ y_{0}\leq Y\leq y}.\]
Summing over all possibilities of \(Y\) yields that (recall that \(j=\lfloor \frac{n}{n_{0}}\rfloor\))
\begin{align*}
  E^{\omega,T}(N_{n,1})&\leq \sum_{l=jn_{0}}^{n}\sum_{|Y|=l} P^{\omega,T}(\tau_{Y}<\infty)\frac{\sum_{|y|=n,Y\leq y}V_{y,Y}}{P_{Y}^{\omega,T}(\tau_{Y}^{*}=\infty)}\mathds{1}_{Z^T(Y,n_{0})>K_{0}}\\
&\leq  \sum_{l=jn_{0}}^{n}\sum_{|Y|=l} P^{\omega,T}(\tau_{\parent{Y}}<\infty)\frac{\sum_{|y|=n,Y\leq y}V_{y,Y}}{P_{Y}^{\omega,T}(\tau_{Y}^{*}=\infty,\tau_{\parent{Y}}=\infty)}\mathds{1}_{Z^T(Y,n_{0})>K_{0}},
\end{align*}
where the last inequality holds because \(P^{\omega,T}(\tau_{Y}<\infty)\leq P^{\omega,T}(\tau_{\parent{Y}}<\infty)\) and \(P_{Y}^{\omega,T}(\tau_{Y}^{*}=\infty)\geq P_{Y}^{\omega,T}(\tau_{Y}^{*}=\infty,\tau_{\parent{Y}}=\infty)\). Summing over the value of \(\parent{Y}\) yields that
\begin{align*}
  E^{\omega,T}(N_{n,1})\leq \sum_{l=jn_{0}-1}^{n-1}\sum_{|x|=l}P^{\omega,T}(\tau_{x}<\infty)\sum_{Y:\parent{Y}=x}\frac{\sum_{|y|=n,Y\leq y}V_{y,Y}}{P_{Y}(\tau_{Y}^{*}=\infty,\tau_{\parent{Y}}=\infty)}\mathds{1}_{Z^T(Y,n_{0})>K_{0}}.
\end{align*}
As conditionally on \(T\), \(P^{\omega,T}(\tau_{x}<\infty) \) and \(\sum_{Y:\parent{Y}=x}\frac{\sum_{|y|=n,Y\leq y}V_{y,Y}}{P_{Y}(\tau_{Y}^{*}=\infty,\tau_{\parent{Y}}=\infty)}\mathds{1}_{d(Y,n_{0})>K_{0}} \) are independent,
\begin{align*}
  \e(N_{n,1})&\leq \e\left(  \sum_{l=jn_{0}-1}^{n-1}\sum_{|x|=l} \e^{T}(P^{\omega,T}(\tau_{x}<\infty)  ) \e^{T}\Big(\sum_{Y:\parent{Y}=x}\frac{\sum_{|y|=n,Y\leq y}V_{y,Y}}{P_{Y}^{\omega,T}(\tau_{Y}^{*}=\infty,\tau_{\parent{Y}}=\infty)}\mathds{1}_{Z^T(Y,n_{0})>K_{0}}  \Big)\right)\\
&=\sum_{l=jn_{0}-1}^{n-1} \e(\sum_{|x|=l} \mathds{1}_{\tau_{x}<\infty} ) \e\Big(\sum_{|Y|=1}\frac{\sum_{|y|=n-l,Y\leq y}V_{y,Y}}{P_{Y}^{\omega,T}(\tau_{Y}^{*}=\infty,\tau_{\parent{Y}}=\infty)}\mathds{1}_{Z^T(Y,n_{0})>K_{0}}\Big).
\end{align*}
Note that for any \(| Y|=1\), \(\frac{\sum_{|y|=n-l,Y\leq y}V_{y,Y}}{P_{Y}^{\omega,T}(\tau_{Y}^{*}=\infty,\tau_{\parent{Y}}=\infty)}\mathds{1}_{Z^T(Y,n_{0})>K_{0}}\) are i.i.d. By Lemma~\ref{vrjpGW-bdT},
\begin{equation}
  \label{vrjpGW-Nn1A}
  \e(N_{n,1})\leq bc_{11}\sum_{l=jn_{0}-1}^{n-1}\mathcal{A}_{n-l}
\end{equation}
where \[\mathcal{A}_{n-l}=\e\Big(\frac{\sum_{|y|=n-l-1}V_{y,\rho}}{P^{\omega,T}(\tau_{\rho}^{*}=\infty,\tau_{\parent{\rho}}=\infty)}\mathds{1}_{Z^T(\rho,n_{0})>K_{0}}\Big).\]
By Cauchy-Schwartz inequality,
\begin{align*}
  \mathcal{A}_{n-l}\leq \sqrt{\e\bigg[\Big(\sum_{|y|=n-l-1}V_{y,\rho}\Big)^{2}\bigg]\e\Big[\frac{\mathds{1}_{Z^T(\rho,n_{0})>K_{0}}}{P^{\omega,T}(\tau_{\rho}^{*}=\infty,\tau_{\parent{\rho}}=\infty)^{2}}\Big]}
\end{align*}
Recall that \(Z_{n}^{T}\) denotes the number of vertices at the \(n\)-th generation of the tree \(T\), using Lemma~\ref{vrjpGW-peu} then applying again Cauchy-Schwartz inequality to \(\Big(\sum_{|y|=n-l-1}V_{y,\rho}\Big)^{2}\) implies that
\begin{align*}
  \mathcal{A}^{2}_{n-l}&\leq c_{14}^{n_0}\e\Big(Z^{T}_{n-l-1}\sum_{|y|=n-l-1}V_{y,\rho}^{2}\Big)\\
  &\leq c_{19}\E_{GW}[c_{20}^{n-l-1}\Big(Z^T_{n-l-1}\Big)^2],
\end{align*}
where the second inequality follows from \(\e^T[V_{y,\rho}]\leq c_{20}^{|y|}\). Plugging it into~\eqref{vrjpGW-Nn1A} implies that
\[\e(N_{n,1})\leq bc_{11}\sqrt{c_{19}}\sum_{l=jn_{0}-1}^{n-1} \sqrt{\E_{GW}[c_{20}^{n-l-1}\Big(Z^T_{n-l-1}\Big)^2]}\leq c_{21}\sum_{k=0}^{n_0}\sqrt{c_{20}^k\E_{GW}\Big[(Z^T_k)^2\Big]}\leq c_{22},\]
since \(\E_{GW}[(Z^T_1)^2]<\infty\).
Analoguesly, for \(N^*_{n,1}\) we get that
\[E^{\omega,T}(N^*_{n,1})\leq \sum_{l=jn_{0}-1}^{n-1}\sum_{|x|=l}P^{\omega,T}(\tau_{x}<\infty)\sum_{Y:\parent{Y}=x}\frac{\sum_{|y|=n,Y\leq y}V_{y,Y}\frac{A_{y}}{1+A_{y}B_{y}}}{P_{Y}^{\omega,T}(\tau_{Y}^{*}=\infty,\tau_{\parent{Y}}=\infty)}\mathds{1}_{Z^T(Y,n_{0})>K_{0}}.\]
And recounting on the same arguments gives a finite upper bound for \(\e[N^*_{n,1}]\).
\end{proof}
\begin{proof}[Proof of Lemma~\ref{vrjpGW-boundwithouW2}]
Again we only give the proof for \(\e(N_{n,2}^{\lambda})\). For \(y\in T\), as \(Z^T(y_{0},n_{0})\leq K_{0}\) and \(y_{0}\notin \mathcal{W}\), we can find the youngest ancestor \(Y_{1}\) of \(y\) in \(T_{n_{0}}\) such that \(Z^T(Y_{1},n_{0})>K_{0}\), automatically \(Y_{1}<y_{0}\). Let \(Y_{2}\) be the youngest descendant of \(Y_{1}\) in \(T_{n_{0}}\) such that it is an ancestor of \(y\). Let \(Y_{3}\) be the youngest descendant of \(y\) in \(T_{n_{0}}\) such that \(Z^T(Y_{3},n_{0})>K_{0}\).
\begin{figure}[!h]
  \centering
  \includegraphics[width=.25\textwidth]{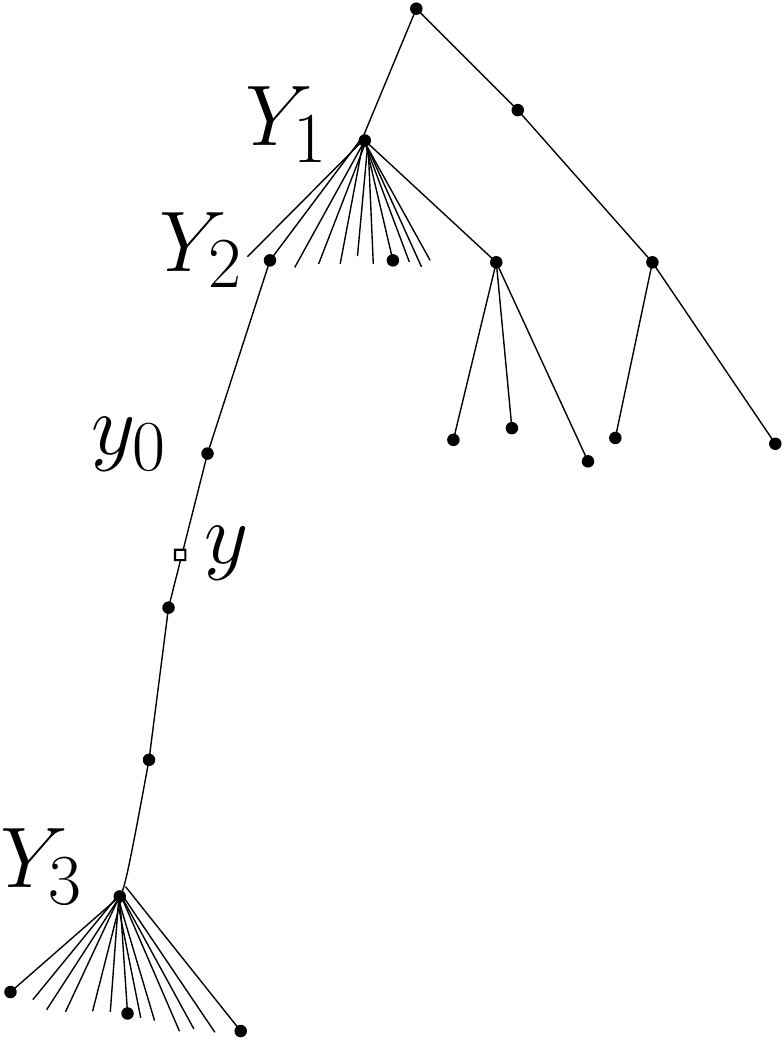}
  \caption{An example of \(Y_1,Y_2,Y_3\).}
  \label{vrjpGW-Y123}
\end{figure}

\noindent For any \(0<\lambda\leq 1\),
\begin{align}
E^{\omega. T}[N_{n,2}^\lambda]&\leq E^{\omega, T}\Big[\sum_{|y|=n}N(y)^\lambda\mathds{1}_{Z^T(y_{0},n_{0})\leq K_{0},\ y_{0}\notin \mathcal{W}}\Big]\nonumber\\
&\leq\sum_{|y|=n}\mathds{1}_{Z^T(y_{0},n_{0})\leq K_{0},\ y_{0}\notin \mathcal{W}}P^{\omega, T}(\tau_y<\infty)\Big(E^{\omega, T}_y[N(y)]\Big)^\lambda.\label{vrjpGW-EwN2}
\end{align}

In what follows, we identify \(\widetilde{P}^{\omega}\) with the distribution of a one-dimensional random walk \(\widetilde{\eta}\) on the path \([\![\parent{Y_1}, \ Y_3]\!]\). Let us state the following lemmas which will be used in~\eqref{vrjpGW-EwN2}.
\begin{lem}For any \(y\in T\) such that \(Y_1<Y_2<y< Y_3\), let \(y^{*}\) be the unique child of \(y\) which is also ancestor of \(Y_3\). Then,
  \label{vrjpGW-bdEyNy}
  \begin{equation}
    \label{vrjpGW-EwNy}
    \Big(E^{\omega, T}_y[N(y)]\Big)^\lambda\leq \Big(\frac{1+A_{y}B_{y}}{1+A_{y}A_{y^{*}}}\Big)^\lambda\widetilde{G}^{\tilde{\tau}_{Y_{1}}\wedge  \tilde{\tau}_{Y_{3}}}(y,y)^\lambda \frac{2}{p_{1}(Y_{1},Y_{2})P_{Y_{3}}^{\omega,T}(\tau_{Y_{3}}^{*}=\infty,\tau_{\parent{Y_{3}}}=\infty)}
  \end{equation}
 where \(\widetilde{G}^{\tilde{\tau}_{Y_{1}}\wedge \tilde{\tau}_{Y_{3}}}(y,y)=\widetilde{E}_{y}^{\omega}\Big(\sum_{k=0}^{\tilde{\tau}_{Y_{1}}\wedge   \tilde{\tau}_{Y_{3}}}\mathds{1}_{\widetilde{\eta}_{k}=y}\Big)\) is the Green function associated with \((\widetilde{\eta}_n)\).
\end{lem}
\begin{lem}
  \label{vrjpGW-bdPTy}
  \begin{equation}
    \label{vrjpGW-PwTy}
    P^{\omega,T}(\tau_{y}<\infty)\leq P^{\omega,T}(\tau_{Y_{1}}<\infty)\widetilde{P}_{Y_{1}}^{\omega}(\tilde{\tau}_{y}< \tilde{\tau}_{Y_{1}-1})^\lambda\frac{1}{p_{1}(Y_{1},Y_{2})}.
  \end{equation}
\end{lem}
The proofs of Lemmas~\ref{vrjpGW-bdEyNy} and~\ref{vrjpGW-bdPTy} can be found in section 5.2 of~\cite{aidekon2008transient} with slight modifications, so we feel free to omit them (see (5.10) and (5.11) therein).
 Now plugging~\eqref{vrjpGW-EwNy} and~\eqref{vrjpGW-PwTy} into~\eqref{vrjpGW-EwN2} yields that
\begin{equation*}
E^{\omega,T}(N_{n,2}^\lambda)\leq \sum_{|y|=n} \frac{2P^{\omega,T}(\tau_{Y_{1}}<\infty)}{p_{1}(Y_{1},Y_{2})^2P_{Y_{3}}^{\omega,T}(\tau_{Y_{3}}^{*}=\infty,\tau_{\parent{Y_{3}}}=\infty)}\bigg(\frac{1+A_{y}B_{y}}{1+A_{y}A_{y^{*}}}\widetilde{P}_{Y_{1}}^{\omega}(\tilde{\tau}_{y}< \tilde{\tau}_{Y_{1}-1})\widetilde{G}^{\tilde{\tau}_{Y_{1}}\wedge  \tilde{\tau}_{Y_{3}}}(y,y)\bigg)^\lambda.
\end{equation*}
By Lemma~\ref{vrjpGW-RWRE1}, one sees that
\begin{align*}
&E^{\omega,T}(N_{n,2}^\lambda)\leq \sum_{|y|=n} \frac{2P^{\omega,T}(\tau_{Y_{1}}<\infty)}{p_{1}(Y_{1},Y_{2})^2P_{Y_{3}}^{\omega,T}(\tau_{Y_{3}}^{*}=\infty,\tau_{\parent{Y_{3}}}=\infty)}\bigg(\frac{1+A_{y}B_{y}}{1+A_{y}A_{y^{*}}}\widetilde{E}^\omega_{Y_1}[\tilde{\tau}_{\parent{Y_1}}\wedge \tilde{\tau}_{Y_3}]\bigg)^\lambda\\
&\leq \sum_{|y|=n} \frac{2P^{\omega,T}(\tau_{Y_{1}}<\infty)}{p_{1}(Y_{1},Y_{2})^2P_{Y_{3}}^{\omega,T}(\tau_{Y_{3}}^{*}=\infty,\tau_{\parent{Y_{3}}}=\infty)}\Big(\frac{1+A_{y}B_{y}}{1+A_{y}A_{y^{*}}}\Big)^\lambda S_{\lambda, [\![  Y_1, Y_2 ]\!]}\bigg(1+A_{Y^*_{2}}^{\lambda}\Big(1+\widetilde{E}_{Y^*_{2}}^{\omega}[\tilde{\tau}_{Y_{2}}\wedge  \tilde{\tau}_{Y_{3}}]^\lambda\Big)\bigg)
\end{align*}
where \(Y_{2}^{*}\) is the children of \(Y_{2}\) along \([\![Y_{2},Y_{3}]\!]\). Decompose the sum over \(|y|=n\) by
\[\sum_{|y|=n} = \sum_{y: |y|=n, Y_1=\rho}+\sum_{l=1}^{(j-1)}\sum_{|x|=ln_0-1}\sum_{y:\parent{Y_1}=x, |y|=n}.\]
We get that
\begin{align*}
  &E^{\omega,T}(N_{n,2}^\lambda)\leq \sum_{|y|=n,Y_{1}=\rho} \frac{2S_{\lambda, [\![  \rho, Y_2 ]\!]} }{p_{1}(\rho,Y_{2})^{2}P_{Y_{3}}^{\omega,T}(\tau_{Y_{3}}^{*}=\infty,\tau_{\parent{Y_{3}}}=\infty) } \Theta_\lambda(Y_2,y,Y_3)\\
&+\sum_{l=1}^{j-1}\sum_{|x|=l n_{0}-1}\sum_{|y|=n,\parent{Y_{1}}=x}  \frac{ 2P^{\omega,T}(\tau_{Y_{1}}<\infty)S_{\lambda, [\![  Y_1, Y_2 ]\!]} }{p_{1}(Y_{1},Y_{2})^{2}P_{Y_{3}}^{\omega,T}(\tau_{Y_{3}}^{*}=\infty,\tau_{\parent{Y_{3}}}=\infty) }\Theta_\lambda(Y_2,y,Y_3),
\end{align*}
where 
\[
\Theta_\lambda(Y_2,y,Y_3):=\left(\frac{1+A_{y}B_{y}}{1+A_{y}A_{y^*}}\right)^\lambda \bigg(1+A^\lambda_{Y^*_{2}}\Big(1+\widetilde{E}_{Y^*_{2}}^{\omega}[\tilde{\tau}_{Y_{2}}\wedge  \tilde{\tau}_{Y_{3}}]^\lambda\Big)\bigg).
\]

Given the GW tree \(T\), note that \(S_{\lambda,[\![ Y_1, Y_2 ]\!]}\in\sigma\{A_{z}; Y_1\leq z\leq Y_2\}\), \(p_1(\rho,Y_2)\in\sigma\{A_u: u\in (T\setminus T_{Y_2})\cup\{Y_2\}\}\), \(P_{Y_{3}}^{\omega,T}(\tau_{Y_{3}}^{*}=\infty,\tau_{\parent{Y_{3}}}=\infty)\in\sigma\{A_u; u\in T_{Y_3}\}\) and \(\Theta_\lambda(Y_2, y,Y_3)\in \sigma\{A_u; Y_2<u\leq Y_3\}\). Therefore,
\begin{align}\label{vrjpGW-ETN2}
&\e^T[N_{n,2}^\lambda] \leq\sum_{|y|=n,Y_{1}=\rho} \e^T\Big[\frac{2S_{\lambda, [\![  \rho, Y_2 ]\!]} }{p_{1}(\rho,Y_{2})^{2}}\Big]\e^T\Big[\frac{ \Theta_\lambda(Y_2,y,Y_3)\mathds{1}_{Z^T(Y_3,n_0)>K_0}}{P_{Y_{3}}^{\omega,T}(\tau_{Y_{3}}^{*}=\infty,\tau_{\parent{Y_{3}}}=\infty) }\Big]\nonumber\\
&+\sum_{l=1}^{j-1}\sum_{|x|=l n_{0}-1}\sum_{|y|=n,\parent{Y_{1}}=x}  \e^T\Big[\frac{ 2P^{\omega,T}(\tau_{Y_{1}}<\infty)S_{\lambda, [\![  Y_1, Y_2 ]\!]} }{p_{1}(Y_{1},Y_{2})^{2}}\Big]\e^T\Big[\frac{ \Theta_\lambda(Y_2,y,Y_3)\mathds{1}_{Z^T(Y_3,n_0)>K_0}}{P_{Y_{3}}^{\omega,T}(\tau_{Y_{3}}^{*}=\infty,\tau_{\parent{Y_{3}}}=\infty) }\Big].
\end{align}
Observe that
\begin{align*}
P_{Y_{3}}^{\omega,T}(\tau_{Y_{3}}^{*}=\infty,\tau_{\parent{Y_{3}}}=\infty)\geq p_1(Y_3, u)\mathds{1}_{Y_3<u, |u|=|Y_3|+n_0}.
\end{align*}
\begin{multline*}
\e^T\Big[\frac{ \Theta_\lambda(Y_2,y,Y_3)\mathds{1}_{Z^T(Y_3,n_0)>K_0}}{P_{Y_{3}}^{\omega,T}(\tau_{Y_{3}}^{*}=\infty,\tau_{\parent{Y_{3}}}=\infty) }\Big\vert A_u,  Y_2<u\leq Y_3\Big]=\Theta_\lambda(Y_2,y,Y_3)\e\bigg[\frac{\mathds{1}_{Z^T(Y_3,n_0)>K_0}}{P_{Y_{3}}^{\omega,T}(\tau_{Y_{3}}^{*}=\infty,\tau_{\parent{Y_{3}}}=\infty)}\Big\vert A_{Y_3}\bigg]\\
\leq\Theta_\lambda(Y_2,y,Y_3)\e\bigg[\mathds{1}_{Z^T(Y_3,n_0)>K_0}\sum_{u: Y_3<u, |u|=|Y_3|+n_0}\frac{1}{p_1(Y_3, u)}\Big\vert A_{Y_3}\bigg].
\end{multline*}
Applying Lemma~\ref{vrjpGW-peu} to the subtree rooted at \(Y_3\) implies that
\[
\e^T\Big[\frac{ \Theta_\lambda(Y_2,y,Y_3)\mathds{1}_{Z^T(Y_3,n_0)>K_0}}{P_{Y_{3}}^{\omega,T}(\tau_{Y_{3}}^{*}=\infty,\tau_{\parent{Y_{3}}}=\infty) }\Big]\leq c_{23}\e^T\Big[(1+\frac{1}{A_{Y_3}})\Theta_\lambda(Y_2,y,Y_3)\Big].
\]
Plugging it into~\eqref{vrjpGW-ETN2} implies that
\begin{equation*}
\e^T[N_{n,2}^\lambda] \leq \Delta_1(n)+\Delta_2(n),
\end{equation*}
where
\begin{align}
\Delta_1(n):=& 2c_{23}\sum_{|y|=n,Y_{1}=\rho} \e^T\Big[\frac{S_{\lambda, [\![  \rho, Y_2 ]\!]} }{p_{1}(\rho,Y_{2})^{2}}\Big]\e^T\Big[(1+\frac{1}{A_{Y_3}})\Theta_\lambda(Y_2,y,Y_3)\Big]\label{vrjpGW-delta1}\\
\Delta_2(n):=& 2c_{23}\sum_{l=1}^{j-1}\sum_{|x|=l n_{0}-1}\sum_{|y|=n,\parent{Y_{1}}=x}  \e^T\Big[\frac{ P^{\omega,T}(\tau_{Y_{1}}<\infty)S_{\lambda, [\![  Y_1, Y_2 ]\!]} }{p_{1}(Y_{1},Y_{2})^{2}}\Big]\e^T\Big[(1+\frac{1}{A_{Y_3}})\Theta_\lambda(Y_2,y,Y_3)\Big]\label{vrjpGW-delta2}.
\end{align}
So,
\begin{equation}\label{vrjpGW-sumNn2}
\e[N_{n,2}^\lambda]\leq \E_{\Q}[\Delta_1(n)+\Delta_2(n)].
\end{equation}

We firstly bound \(\Delta_1(n)\), note that (since \(\lambda\leq 1\))
\[
\Big(\frac{1+A_{y}B_{y}}{1+A_{y}A_{y^*}}\Big)^\lambda\leq \Big( 1+\frac{\sum_{z:\parent{z}=y, z\neq y^*}A_z}{A_{y^*}}\Big)^\lambda\leq 1+\frac{\sum_{z:\parent{z}=y, z\neq y^*}A_z^\lambda}{A_{y^*}^\lambda},
\]
with \(\sum_{z:\parent{z}=y, z\neq y^*}1\leq K_0\). If \(|Y_2|=mn_0<n\), \(|Y_3|=(m+k)n_0>n\), by Markov property and the fact that \(\{A_{z}, \parent{z}=y,z\neq y^{*}\}\) is independent of \(\{A_{z},z\in [\![Y_{2},Y_{3}]\!]:=[\![-1,kn_{0}-1]\!]\}\),
\begin{align*}
&\e^T\Big[(1+\frac{1}{A_{Y_3}})\Theta_\lambda(Y_2,y,Y_3)\Big]\\
&\leq\e^{T}\left[(1+\frac{1}{A_{kn_{0}-1}})(1+\frac{\sum_{z:\parent{z}=y, z\neq y^*}A_z^{\lambda}}{A_{n-mn_{0}}^{\lambda}} )(1+A_{0}^{\lambda}(1+\tilde{E}_{0}^{\omega}(\tilde{\tau}_{-1}\wedge  \tilde{\tau}_{kn_0-1} )^\lambda))\right]\\
&\leq c_{24}+c_{24}\E\Big((1+\frac{1}{A_{kn_0-1}})(1+\frac{1}{A^\lambda_{n-mn_0}})A_0^\lambda\widetilde{E}^\omega_0[\tilde{\tau}_{-1}\wedge  \tilde{\tau}_{kn_0-1}]^\lambda\Big).
\end{align*}
Now apply Lemma~\ref{vrjpGW-bdtau}, we have
\begin{equation}\label{vrjpGW-bdtheta}
\e^T\Big[(1+\frac{1}{A_{Y_3}})\Theta_\lambda(Y_2,y,Y_3)\Big]\leq c_{25}(q_1+\delta)^{-|Y_3|+|Y_2|+1}.
\end{equation}
Applying Cauchy-Schwartz inequality to \(\e^T\Big[\frac{S_{\lambda, [\![  \rho, Y_2 ]\!]} }{p_{1}(\rho,Y_{2})^{2}}\Big]\) yields
\begin{align*}
\Delta_1(n)&\leq c_{23} \sum_{|y|=n,Y_{1}=\rho} 2\bigg(\sqrt{ \e^T\Big[S_{\lambda, [\![  \rho, Y_2 ]\!]}^2\Big]\e^T\Big[\frac{1}{p_1(\rho,Y_2)^4}\Big]}\bigg)\e^T\Big[(1+\frac{1}{A_{Y_3}})\Theta_\lambda(Y_2,y,Y_3)\Big]\\
&\leq c_{26} \sum_{|y|=n,Y_{1}=\rho} \sqrt{\e^T\Big[\frac{1}{p_1(\rho,Y_2)^4}\Big]}\e^T\Big[(1+\frac{1}{A_{Y_3}})\Theta_\lambda(Y_2,y,Y_3)\Big],
\end{align*}
where the last inequality holds because \(\e^T\Big[S_{\lambda,[\![  \rho, Y_2 ]\!]}^2\Big]\leq c_{27}(n_0)<\infty\). By~\eqref{vrjpGW-bdtheta},
\begin{align*}
\Delta_1(n)&\leq c_{28}\sum_{|y|=n,Y_{1}=\rho} \e^T\Big[\frac{1}{p_1(\rho,Y_2)^4}\Big](q_1+\delta)^{-|Y_3|+|Y_2|+1}\\
&=c_{28}\e^T\Big[\sum_{|u|=n_0} \mathds{1}_{Z_{n_0}^T>K_0}\frac{1}{p_1(\rho,u)^4}\Big]\sum_{y: |y|=n, Y_2=u}(q_1+\delta)^{-|Y_3|+n_0+1}
\end{align*}
Observe that
\[
\sum_{y: |y|=n, Y_2=u}(q_1+\delta)^{-|Y_3|+n_0+1}\leq \sum_{z: |z|>n, z\in \mathcal{W}(T_u) }(q_1+\delta)^{-|z|+n_0+1}.
\]
Hence,
\begin{align*}
\Delta_1(n)&\leq c_{28}\e^T\Big[\sum_{|u|=n_0} \mathds{1}_{Z_{n_0}^T>K_0}\frac{1}{p_1(\rho,u)^4}\Big]\sum_{z: |z|>n, z\in \mathcal{W}(T_u) }(q_1+\delta)^{-|z|+n_0+1}.
\end{align*}
Taking expectation under \(GW(dT)\) implies that
\[
\E_{\Q}[\Delta_1(n)]\leq c_{28}\e\Big[\sum_{|u|=n_0} \mathds{1}_{Z_{n_0}^T>K_0}\frac{1}{p_1(\rho,u)^4}\Big]\E_{\Q}\Big[\sum_{z: |z|>n-n_0, z\in \mathcal{W} }(q_1+\delta)^{-|z|+1}\Big],
\]
which by Lemma~\ref{vrjpGW-peu} is bounded by
\[
c_{29}\E_{\Q}\Big[\sum_{z: |z|>n-n_0, z\in \mathcal{W} }(q_1+\delta)^{-|z|+1}\Big]=c_{29}\sum_{l>n/n_0-1}\E_{\Q}\Big[\sum_{|z|=ln_0, z\in\mathcal{W}}(q_1+\delta)^{-|z|+1}\Big].
\]
Recall that \(\mathcal{W}\) is a GW tree of mean \(\E[Z_{n_0}; Z_{n_0}\leq K_0]\leq r^{n_0}\). We can choose \(r\) to be \(q_1+\delta/2\) so that
\[
\sum_{l\geq 1}\E_{\Q}\Big[\sum_{|z|=ln_0, z\in\mathcal{W}}(q_1+\delta)^{-|z|+1}\Big]\leq \sum_{l\geq 1} (q_1+\delta)^{-ln_0+1}r^{ln_0}<c_{30}\gamma^{l_0},
\]
where \(\gamma:=(\frac{q_1+\delta/2}{q_1+\delta})^{n_0}<1\) and \(l_0:=\lceil \frac{n}{n_0}\rceil-1=j-1\).
As a result, for any \(n> n_0\),
\begin{equation}\label{vrjpGW-bddelta1}
\E_{\Q}[\Delta_1(n)]\leq c_{31}\gamma^{l_0}<\infty.
\end{equation}
Turn to \(\Delta_2(n)\). As \(P^{\omega,T}(\tau_{Y_{1}}<\infty)\leq P^{\omega,T}(\tau_{\parent{Y_{1}}}<\infty)\), one sees that
\begin{align*}
\Delta_2(n)\leq& 2c_{23}\sum_{l=1}^{j-1}\sum_{|x|=l n_{0}-1}\sum_{|y|=n,\parent{Y_{1}}=x}  \e^T\Big[\frac{ P^{\omega,T}(\tau_{x}<\infty)S_{\lambda, [\![  Y_1, Y_2 ]\!]} }{p_{1}(Y_{1},Y_{2})^{2}}\Big]\e^T\Big[(1+\frac{1}{A_{Y_3}})\Theta_\lambda(Y_2,y,Y_3)\Big],
\end{align*}
which equals to
\begin{align*}
\sum_{l=1}^{j-1}\sum_{|x|=l n_{0}-1}\sum_{z: \parent{z}=x}\p^T(\tau_{x}<\infty)2c_{23}\sum_{|y|=n,Y_1=z}  \e^T\Big[\frac{ S_{\lambda, [\![  Y_1, Y_2 ]\!]} }{p_{1}(Y_{1},Y_{2})^{2}}\Big]\e^T\Big[(1+\frac{1}{A_{Y_3}})\Theta_\lambda(Y_2,y,Y_3)\Big],
\end{align*}
as \(P^{\omega,T}(\tau_{x}<\infty)\) and \(\frac{S_{\lambda, [\![  Y_1, Y_2 ]\!]} }{p_{1}(Y_{1},Y_{2})^{2}}\) are independent under \(\p^T\). 

Note that  for all \(z\in T\),
\(2c_{23}\sum_{|y|=n,Y_1=z}  \e^T\Big[\frac{ S_{\lambda,[\![  Y_1, Y_2 ]\!]} }{p_{1}(Y_{1},Y_{2})^{2}}\Big]\e^T\Big[(1+\frac{1}{A_{Y_3}})\Theta_\lambda(Y_2,y,Y_3)\Big]\) are i.i.d.\ copies of \(\Delta_1(n-|z|)\). Taking expectation yields that
\begin{align*}
\E_{\Q}[\Delta_2(n)]&\leq \sum_{l=1}^{j-1} \e\Big[\sum_{|x|=ln_0-1}\mathds{1}_{\tau_x<\infty}(d(x)-1)\Big]\E_{\Q}[\Delta_1(n-ln_0)]\\
&\leq b c_{31} \sum_{l=1}^{j-1} \e\Big[\sum_{|x|=ln_0-1}\mathds{1}_{\tau_x<\infty}\Big]\gamma^{j-l-1},
\end{align*}
where the last inequality follows from~\eqref{vrjpGW-bddelta1}. By Lemma~\ref{vrjpGW-bdT}, for any \(j\geq2\), 
\begin{align*}
\E_{\Q}[\Delta_2(n)]&\leq  c_{32} \sum_{l=1}^{j-1} \gamma^{j-1-l}\leq c_{33}<\infty.
\end{align*}
Plugging the above inequality and~\eqref{vrjpGW-bddelta1} into~\eqref{vrjpGW-sumNn2} implies that
\[
\e[N_{n,2}^\lambda]\leq \E_{\Q}[\Delta_1(n)]+\E_{\Q}[\Delta_2(n)]<\infty.
\]
The estimate of \(\e[(N_{n,2}^*)^\lambda]\) follows from similar arguments. We feel free to omit it.
\end{proof}

\appendix
\section{Proofs of one dimensional results}
\label{vrjpGW-1dappendix}
\begin{proof}[Proof of Lemma~\ref{vrjpGW-oneMAMA}]
For any \(i\geq 1\), let \(S_i=-\sum_{j=1}^i\log(A_jA_{j-1})\) and define \(S_0=0\). As \(i\mapsto \tilde{P}_i^{\omega}(\tilde{\tau}_{-1}> \tilde{\tau}_n)\) is the solution to the Dirichlet problem
\[
\begin{cases}
  \varphi(-1)=0,\ \varphi(n)=1\\
  \tilde{E}^{\omega}_i(\varphi(\tilde{\eta}_1))=\varphi(i) & i\in [\![ 0,n-1]\!].
\end{cases}
\]
It follows that
\begin{equation}
\label{vrjpGW-eq-solDirich-1d}
\tilde{P}_{i}^{\omega}(\tilde{\tau}_{-1}>\tilde{\tau}_{n})=\frac{\sum_{j=0}^{i}\exp(S_{j})}{\sum_{j=0}^{n}\exp(S_{j})}.
\end{equation}
As a consequence, for any \(0\leq l\leq n\),
\begin{align*}
  &\tilde{P}_0^{\omega}(\tilde{\tau}_l<\tilde{\tau}_{-1})=\frac{1}{\sum_{j=0}^l\exp(S_j)}\geq \frac{\exp(-\max_{0\leq j\leq l}S_j)}{l+1}\\
  &\tilde{P}_{l+1}^{\omega}(\tilde{\tau}_n<\tilde{\tau}_l)=\frac{\exp(S_{l+1})}{\sum_{j=l+1}^{n}\exp(S_{j})}\leq \exp(-\max_{l+1\leq j\leq n}(S_j-S_{l+1}))\\
  &\tilde{P}_{l-1}^{\omega}(\tilde{\tau}_{-1}<\tilde{\tau}_l)=\frac{\exp(S_{l})}{\sum_{j=0}^l\exp(S_j)}\leq \exp(-\max_{0\leq j\leq l}(S_j-S_l)).
\end{align*}
We only need to consider \(n\) large, take \(l=\lfloor z_1n\rfloor\), note that
\begin{align*}
\tilde{P}_l^{\omega}(\tilde{\tau}_l^*>\tilde{\tau}_{-1}\wedge \tilde{\tau}_n)&=p(l,l+1)\tilde{P}_{l+1}^{\omega}(\tilde{\tau}_n<\tilde{\tau}_l)+p(l,l-1)\tilde{P}_{l-1}^{\omega}(\tilde{\tau}_{-1}<\tilde{\tau}_l)\\
&\leq \max(\tilde{P}_{l+1}^{\omega}(\tilde{\tau}_n<\tilde{\tau}_l),\tilde{P}_{l-1}^{\omega}(\tilde{\tau}_{-1}<\tilde{\tau}_l)).
\end{align*}
Therefore,
\begin{align*}
  \tilde{P}_0^{\omega}(\tilde{\tau}_n\wedge \tilde{\tau}_{-1}>m)&\geq \tilde{P}_0^{\omega}(\tilde{\tau}_l<\tilde{\tau}_{-1})\tilde{P}_l^{\omega}(\tilde{\tau}_l^*<\tilde{\tau}_{-1}\wedge \tilde{\tau}_n)^m\\
&\geq \frac{\exp(-\max_{0\leq j\leq l}S_j)}{l+1}\left(1-\tilde{P}_l^{\omega}(\tilde{\tau}_l^*\geq \tilde{\tau}_{-1}\wedge \tilde{\tau}_n)\right)^m\\
&\geq \frac{\exp(-\max_{0\leq j\leq l}S_j)}{l+1} \left(1-\exp(-\max_{l+1\leq k\leq n}(S_k-S_{l+1})\wedge \max_{0\leq k\leq l}(S_k-S_l))\right)^m\\
&\geq \frac{\mathds{1}_{\max_{0\leq k\leq l}S_k\leq 0}}{l+1} (1-e^{-zn})^m \mathds{1}_{\max_{l+1\leq k\leq n}(S_k-S_{l+1})\geq zn}\mathds{1}_{\max_{0\leq k\leq l}(S_k-S_l)\geq zn}.
\end{align*}
As \(m\approx e^{zn}\), we have \((1-e^{-zn})^m=O(1)\), taking expectation under \(\P(\cdot|A_0\in[a,\frac{1}{a}])\) yields
\begin{align*}
  &\tilde{\p}_0(\tilde{\tau}_n\wedge \tilde{\tau}_{-1}>m|A_0\in [a,\frac{1}{a}])\\
&\geq \frac{c}{n}\P(\max_{0\leq k\leq l}S_k\leq 0,\ \max_{0\leq k\leq l}(S_k-S_l)\geq zn |A_0\in [a,\frac{1}{a}]  )\P( \max_{l+1\leq k\leq n}(S_k-S_{l+1})\geq zn)\\
&\geq \frac{c}{n}\P(\max_{0\leq k\leq l}S_k\leq 0,\ S_l\leq -zn |A_0\in [a,\frac{1}{a}]  )\P( (S_n-S_{l+1})\geq zn).
\end{align*}
For \(k\geq 1\), write \(\mathscr{S}_k=-\sum_{i=1}^k \log A_i\), then as \(S_k=-\log A_0 +\mathscr{S}_{k-1}+\mathscr{S}_k\),
\begin{align*}
  &\P(\max_{0\leq k\leq l}S_k\leq 0,\ S_l\leq -zn |A_0\in [a,\frac{1}{a}] )\\
&\geq \P(A_0\geq 1,A_l\geq 1, \max_{1\leq k\leq l-1}\mathscr{S}_k\leq 0,\ \mathscr{S}_{l-1}\leq -\frac{zn}{2}|A_0\in [a,\frac{1}{a}] )\\
&=\P(A_0\geq 1|A_0\in [a,\frac{1}{a}])\P(A_l\geq 1)\P( \max_{1\leq k\leq l-1}\mathscr{S}_k\leq 0,\ \mathscr{S}_{l-1}\leq -\frac{zn}{2})
\end{align*}
note that
\[\P(\max_{1\leq k\leq l-1}\mathscr{S}_k\leq 0,\ \mathscr{S}_{l-1}\leq -\frac{zn}{2} )\geq \frac{1}{l}\P(\mathscr{S}_{l-1}\leq -\frac{zn}{2})\]
and
\[S_n-S_{l+1}=-\log A_{l+1}-\log A_n-2\sum_{k=l+2}^{n-1}\log A_k.\]
Therefore,
\begin{align*}
  \tilde{\p}_0(\tilde{\tau}_n\wedge \tilde{\tau}_{-1}>m|A_0\in [a,\frac{1}{a}])&\geq \frac{c}{n^2}\P(\mathscr{S}_{l-1}\leq -\frac{zn}{2})\P(S_n-S_{l+1}\geq zn)\\
&\geq \frac{c}{n^2}\P(\mathscr{S}_{l-1}\leq -\frac{zn}{2})\P(A_{l+1}\leq 1)\P(A_n\leq 1)\P(-\sum_{k=l+2}^{n-1}\log A_k\geq \frac{zn}{2})\\
&\geq \frac{c}{n^2}\P(\mathscr{S}_{l-1}\leq -\frac{zn}{2})\P(-\sum_{k=l+2}^{n-1}\log A_k\geq \frac{zn}{2})\\
&\geq \frac{c}{n^{2}}\P(\sum_{k=1}^{l-1}\log A_{k}\geq \frac{zn}{2})\P(\sum_{k=l+2}^{n-1}\log A_{k}\leq -\frac{zn}{2})
\end{align*}
Applying Cram\'er's theorem to sums of i.i.d.\ random variables \(\log A_{k}\), we have
\[\tilde{\p}_0(\tilde{\tau}_n\wedge \tilde{\tau}_{-1}>m|A_0\in [a,\frac{1}{a}])\gtrsim_{n} \exp(-n\left(z_1I(\frac{z}{2z_1})+(1-z_1)I(\frac{-z}{2(1-z_1)})\right))\]
where \(I(x)=\sup_{t\in \mathbb{R}}\{tx-\log \E(A^t)\}\) is the associated rate function.
\end{proof}

\begin{proof}[Proof of Lemma~\ref{vrjpGW-LDP}]
Replace \(I(\frac{-z}{2(1-z_{1})})\) using
\begin{align*}
  I(-x)&=\sup_{t\in \mathbb{R}}\{-tx-\log\E(A^{t})\}=\sup_{t\in \mathbb{R}}\{-tx-\log \E(A^{1-t})\}\\
&=\sup_{s\in \mathbb{R}}\{-(1-s)x-\log \E(A^{s})\}=I(x)-x.
\end{align*}
For fixed \(z\), by convexity of the rate function \(I\), the supremum of \(-z_{1}I(\frac{z}{2z_{1}})-(1-z_{1})I(\frac{z}{2(1-z_{1})})\) is obtained when \(z_1=\frac{1}{2}\), we are left to compute
\[\sup_{0<z}\{\frac{\log q_1-I(z)}{z}+\frac{1}{2}\},\]
clearly, \(\frac{\log q_1-I(z)}{z}\leq -t^*\), when \(z\) is such that \((t\mapsto \log \E(A^t))'(t^*)=z>0\), the maximum is obtained.
\end{proof}

\begin{proof}[Proof of Lemma~\ref{vrjpGW-RWRE1} ]
Observe that
\begin{align*}
&\widetilde{P}^{\omega}_{Y_1}(\tilde{\tau}_y<\tilde{\tau}_{\overleftarrow{Y_1}})\widetilde{G}^{\tilde{\tau}_{Y_1}\wedge \tilde{\tau}_{Y_3}}(y,y)=\widetilde{P}^{\omega}_{Y_1}(\tilde{\tau}_y<\tilde{\tau}_{\overleftarrow{Y_1}}\wedge \tilde{\tau}_{Y_3})\widetilde{E}^{\omega}_y\bigg[\sum_{k=0}^{\tilde{\tau}_{Y_1}\wedge \tilde{\tau}_{Y_3}}1_{\{\widetilde{\eta}_k=y\}}\bigg]\\
\leq& \widetilde{P}^{\omega}_{Y_1}(\tilde{\tau}_y<\tilde{\tau}_{\overleftarrow{Y_1}}\wedge \tilde{\tau}_{Y_3})\widetilde{E}^{\omega}_y\bigg[\sum_{k=0}^{\tilde{\tau}_{\overleftarrow{Y_1}}\wedge \tilde{\tau}_{Y_3}}1_{\{\widetilde{\eta}_k=y\}}\bigg]=\widetilde{E}^{\omega}_{Y_1}\bigg[\sum_{k=0}^{\tilde{\tau}_{\overleftarrow{Y_1}}\wedge \tilde{\tau}_{Y_3}}1_{\{\widetilde{\eta}_k=y\}}\bigg].
\end{align*}
Obviously,
\[
\widetilde{E}^{\omega}_{Y_1}\bigg[\sum_{k=0}^{\tilde{\tau}_{\overleftarrow{Y_1}}\wedge \tilde{\tau}_{Y_3}}1_{\{\widetilde{\eta}_k=y\}}\bigg]\leq \widetilde{E}^{\omega}_{Y_1}[\tilde{\tau}_{\overleftarrow{Y_1}}\wedge \tilde{\tau}_{Y_3}].
\]
This gives us~\eqref{vrjpGW-bdPG}.

Moreover,  to get~\eqref{vrjpGW-eqtilTpTm}, we only need to show that for any \(0\leq p<m\), we have
\begin{equation}\label{vrjpGW-onestepplus}
\widetilde{E}^{\omega}_p[\tilde{\tau}_{p-1}\wedge \tilde{\tau}_m]\leq 1+A_{p}A_{p+1}+A_pA_{p+1}\widetilde{E}^{\omega}_{p+1}[\tilde{\tau}_{p}\wedge \tilde{\tau}_m].
\end{equation}
In fact, since \(0\leq\lambda\leq 1\), \eqref{vrjpGW-onestepplus} implies that
\[\widetilde{E}^{\omega}_p[\tilde{\tau}_{p-1}\wedge \tilde{\tau}_m]^{\lambda}\leq 1+(A_{p}A_{p+1})^{\lambda}+(A_pA_{p+1})^{\lambda}\widetilde{E}^{\omega}_{p+1}[\tilde{\tau}_{p}\wedge \tilde{\tau}_m]^{\lambda}.\]
Applying this inequality a few times along the interval \([\! [Y_1, \, Y_3]\!]\), we obtain~\eqref{vrjpGW-eqtilTpTm}.
It remains to show~\eqref{vrjpGW-onestepplus}. Observe that
\begin{align*}
&\widetilde{E}^{\omega}_p[\tilde{\tau}_{p-1}\wedge \tilde{\tau}_m]=\widetilde{\omega}(p,p-1)+\widetilde{\omega}(p,p+1)(1+\widetilde{E}^{\omega}_{p+1}[\tilde{\tau}_{p-1}\wedge \tilde{\tau}_m])\\
&=1+\widetilde{\omega}(p,p+1)\widetilde{E}^{\omega}_{p+1}[\tilde{\tau}_{p-1}\wedge \tilde{\tau}_m]\\
&=1+\widetilde{\omega}(p,p+1)\Big(\widetilde{E}^{\omega}_{p+1}[\tilde{\tau}_m; \tilde{\tau}_m<\tilde{\tau}_p]+\widetilde{E}^{\omega}_{p+1}[\tilde{\tau}_p; \tilde{\tau}_p< \tilde{\tau}_m]+\widetilde{P}^{\omega}_{p+1}(\tilde{\tau}_p<\tilde{\tau}_m)\widetilde{E}^{\omega}_p[\tilde{\tau}_{p-1}\wedge \tilde{\tau}_m]\Big).
\end{align*}
It follows that
\begin{align*}
&\widetilde{E}^{\omega}_p[\tilde{\tau}_{p-1}\wedge \tilde{\tau}_m]=\frac{1+\widetilde{\omega}(p,p+1)\widetilde{E}^{\omega}_{p+1}[\tilde{\tau}_p\wedge \tilde{\tau}_m]}{1-\widetilde{\omega}(p,p+1)\widetilde{P}^{\omega}_{p+1}(\tilde{\tau}_p<\tilde{\tau}_m)}\\
&=\frac{1+\widetilde{\omega}(p,p+1)\widetilde{E}^{\omega}_{p+1}[\tilde{\tau}_p\wedge \tilde{\tau}_m]}{\widetilde{\omega}(p,p-1)+\widetilde{\omega}(p,p+1)\widetilde{P}^{\omega}_{p+1}(\tilde{\tau}_m<\tilde{\tau}_p)}\leq \frac{1+\widetilde{\omega}(p,p+1)\widetilde{E}^{\omega}_{p+1}[\tilde{\tau}_p\wedge \tilde{\tau}_m]}{\widetilde{\omega}(p,p-1)}.
\end{align*}
Therefore,
\[
\widetilde{E}^{\omega}_p[\tilde{\tau}_{p-1}\wedge \tilde{\tau}_m]\leq (1+A_pA_{p+1})+A_pA_{p+1}\widetilde{E}^{\omega}_{p+1}[\tilde{\tau}_p\wedge \tilde{\tau}_m].
\]

\end{proof}

\begin{proof}[Proof of Lemma~\ref{vrjpGW-bdtau}]
Recall that \(\E[A^t]<\infty\) for any \(t\in\mathbb{R}\). By H\"{o}lder's inequality, it suffices to show that there exists some \(\delta'>0\) such that for all \(n\) large enough,
\begin{equation}\label{vrjpGW-LDPplus}
\E\Big[\Big(\tilde{E}^\omega_0[\tilde{\tau}_{-1}\wedge \tilde{\tau}_{n}]\Big)^{\lambda(1+\delta')}\Big]\leq (q_1+\delta)^{-n}.
\end{equation}
It remains to prove~\eqref{vrjpGW-LDPplus}. In fact, we only need to show that for \(1>\lambda'=\lambda(1+\delta)>0\), 
\begin{equation}\label{vrjpGW-LDPpplus}
\limsup_{n\rightarrow\infty}\frac{\log \E\Big[\Big(\tilde{E}^\omega_0[\tilde{\tau}_{-1}\wedge \tilde{\tau}_{n}]\Big)^{\lambda'}\Big]}{n}\leq \psi(\lambda'+1/2)
\end{equation}
where \(\psi(t)=\log\E(A^{t})\). One therefore sees that if \(t^*-1/2>\lambda'\), then \(\psi(\lambda'+1/2)<\psi(t^*)=-\log q_1\). To show~\eqref{vrjpGW-LDPpplus}, recall that for any \(0\leq i\leq n-1\),
\begin{align*}
\widetilde{G}^{\tilde{\tau}_{-1}\wedge \tilde{\tau}_n}(i,i)&=\widetilde{E}^\omega_i\Big[\sum_{k=0}^{\tilde{\tau}_{-1}\wedge \tilde{\tau}_n}1_{\eta=i}\Big]\\
&=\frac{1}{1-\widetilde{\omega}(i,i-1)\widetilde{P}_{i-1}(\tilde{\tau}_i<\tilde{\tau}_{-1})-\widetilde{\omega}(i,i+1)\widetilde{P}_{i+1}(\tilde{\tau}_i<\tilde{\tau}_n)}.
\end{align*}
Then, \(\widetilde{E}^\omega_{0}[\tilde{\tau}_{-1}\wedge \tilde{\tau}_{n}]=1+\sum_{i=0}^{n-1}\widetilde{P}^\omega_0(\tilde{\tau}_i<\tilde{\tau}_{-1})\widetilde{G}^{\tilde{\tau}_{-1}\wedge \tilde{\tau}_n}(i,i)\) implies that
\[
\widetilde{E}^\omega_{0}[\tilde{\tau}_{-1}\wedge \tilde{\tau}_{n}]=1+\sum_{i=0}^{n-1}\frac{\widetilde{P}^\omega_0(\tilde{\tau}_i<\tilde{\tau}_{-1})}{\widetilde{\omega}(i,i-1)\widetilde{P}^\omega_{i-1}(\tilde{\tau}_{-1}<\tilde{\tau}_i)+\widetilde{\omega}(i,i+1)\widetilde{P}^\omega_{i+1}(\tilde{\tau}_n<\tilde{\tau}_i)}.
\]
Recall that by~\eqref{vrjpGW-eq-solDirich-1d}, if \(S_i:=\sum_{j=1}^i -\log (A_{j-1}A_j)\) for \(i\geq 1\) and \(S_0=0\), then
\begin{align*}
\widetilde{P}^\omega_0(\tilde{\tau}_i<\tilde{\tau}_{-1})&=\frac{1}{\sum_{k=0}^i e^{S_k}}\\
\widetilde{P}^\omega_{i-1}(\tilde{\tau}_{-1}<\tilde{\tau}_i)&=\frac{e^{S_i}}{\sum_{k=0}^i e^{S_k}}\\
\widetilde{P}^\omega_{i+1}(\tilde{\tau}_n<\tilde{\tau}_i)&=\frac{1}{\sum_{k=i+1}^n e^{S_k-S_{i+1}}}.
\end{align*}
 It is immediate that 
\begin{align*}
\frac{\widetilde{P}^\omega_0(\tilde{\tau}_i<\tilde{\tau}_{-1})}{\widetilde{\omega}(i,i-1)\widetilde{P}^\omega_{i-1}(\tilde{\tau}_{-1}<\tilde{\tau}_i)+\widetilde{\omega}(i,i+1)\widetilde{P}^\omega_{i+1}(\tilde{\tau}_n<\tilde{\tau}_i)}=&\frac{\frac{1}{\sum_{k=0}^i e^{S_k}}}{\frac{1}{1+A_iA_{i+1}}\frac{e^{S_i}}{\sum_{k=0}^i e^{S_k}}+\frac{A_iA_{i+1}}{1+A_iA_{i+1}} \frac{1}{\sum_{k=i+1}^n e^{S_k-S_{i+1}}}}\\
\leq&\frac{1}{\frac{1}{1+A_iA_{i+1}}\frac{e^{S_i}}{\sum_{k=0}^i e^{S_k}}+\frac{A_iA_{i+1}}{1+A_iA_{i+1}} \frac{1}{\sum_{k=i+1}^n e^{S_k-S_{i+1}}}}.
\end{align*}
Let \(X_k=-\log A_k\). For any \(0\leq i\leq n\), define
\begin{align*}
&H_i(-X):=\max_{0\leq j\leq i}(-X_j-X_{j+1}-\cdots-X_{i-1}),\\
&H_{n-i-1}(X):=\max_{i+2\leq j\leq n}(X_{i+2}+\cdots+X_{j}).
\end{align*}
Note that
\[
S_k-S_i\leq 2H_{i}(-X)+(-X_i)_+, \forall 0\leq k\leq i,
\]
and that
\[
S_k-S_{i+1}\leq 2H_{n-i-1}(X)+(X_{i+1})_+, \forall i+1\leq k\leq n.
\]
Then, 
\[
\frac{1}{1+A_iA_{i+1}}\frac{e^{S_i}}{\sum_{k=0}^i e^{S_k}}\geq \frac{1}{1+A_iA_{i+1}}\frac{1}{(1+i)e^{2H_i(-X)+(-X_i)_+}}\geq \frac{1}{n(A_i+1)(1+A_iA_{i+1})}e^{-2H_i(-X)}.
\]
Similarly,
\[
\frac{A_iA_{i+1}}{1+A_iA_{i+1}} \frac{1}{\sum_{k=i+1}^n e^{S_k-S_{i+1}}}\geq \frac{(A_{i+1}\wedge 1)A_iA_{i+1}}{n(1+A_iA_{i+1})}e^{-2H_{n-i-1}(X)}.
\]
So,
\begin{multline*}
\frac{1}{1+A_iA_{i+1}}\frac{e^{S_i}}{\sum_{k=0}^i e^{S_k}}+\frac{A_iA_{i+1}}{1+A_iA_{i+1}} \frac{1}{\sum_{k=i+1}^n e^{S_k-S_{i+1}}} \\
\geq\frac{1}{n(A_i+1)(1+A_iA_{i+1})}e^{-2H_i(-X)}+ \frac{(A_{i+1}\wedge 1)A_iA_{i+1}}{n(1+A_iA_{i+1})}e^{-2H_{n-i-1}(X)}\\
\geq \frac{1}{n}\Big(\frac{1}{(A_i\vee1)(1+A_iA_{i+1})}\wedge\frac{(A_{i+1}\wedge 1)A_iA_{i+1}}{1+A_iA_{i+1}}\Big) e^{-2H_i(-X)}\vee e^{-2H_{n-i-1}(X)}.
\end{multline*}
This implies that
\begin{align*}
&\frac{\widetilde{P}^\omega_0(\tilde{\tau}_i<\tilde{\tau}_{-1})}{\widetilde{\omega}(i,i-1)\widetilde{P}^\omega_{i-1}(\tilde{\tau}_{-1}<\tilde{\tau}_i)+\widetilde{\omega}(i,i+1)\widetilde{P}^\omega_{i+1}(\tilde{\tau}_n<\tilde{\tau}_i)}\\
\leq& n\Big((A_i\vee1)(1+A_iA_{i+1})+\frac{1+A_iA_{i+1}}{(A_{i+1}\wedge 1)A_iA_{i+1}}\Big)e^{2H_{i}(-X)\wedge H_{n-i-1}(X)}.
\end{align*}
Thus, for any \(\lambda\leq 1\), \(n\geq2\),
\[ 
\widetilde{E}^\omega_{0}[\tilde{\tau}_{-1}\wedge \tilde{\tau}_{n}]^\lambda\lesssim_{n} n+n^2\sum_{i=0}^{n-1} \Big((A_i\vee1)(1+A_iA_{i+1})+\frac{1+A_iA_{i+1}}{(A_{i+1}\wedge 1)A_iA_{i+1}}\Big)^\lambda e^{2\lambda H_{i}(-X)\wedge H_{n-i-1}(X)}
\]
By independence, 
\begin{equation}\label{vrjpGW-bdpotentiel}
\E\widetilde{E}^\omega_{0}[\tilde{\tau}_{-1}\wedge \tilde{\tau}_{n}]^\lambda\lesssim_{n} n+n^3\max_{0\leq i\leq n-1}\E[e^{2\lambda H_{i}(-X)\wedge H_{n-i-1}(X)}]
\end{equation}
Recall that \(\psi(\lambda)=\log\E[A^\lambda]\) and \(\mathscr{S}_{k}=-\sum_{i=1}^{k}\log A_{i}\). Let \(t>0\), for \(i\geq1\), \(x>0\),
\begin{align}
\P(H_i(-X)\geq xi)&\leq \P(\max_{0\leq k\leq i}[-t \mathscr{S}_{k}-\psi(t)k] \geq xt i-\psi(t)i)\nonumber\\
&\leq \P(\max_{0\leq k\leq i}e^{-t \mathscr{S}_{k}-\psi(t)k}\geq e^{(xt-\psi(t))i})\nonumber\\
&\leq e^{-(xt-\psi(t))i},\label{vrjpGW-bd-X}
\end{align}
where the last inequality stems from Doob's maximal inequality and the fact that \((e^{-t \mathscr{S}_{j}-\psi(t)j})_{j}\) is a martingale. Since \(x\geq \E(\log A)\), \(I(x)=\sup_{t>0}\{tx-\psi(t)\}\), we have
\begin{equation}
\label{vrjpGW-eq-cramer1}
\P(H_i(-X)\geq xi)\leq e^{-I(x)i}.
\end{equation}
Similarly, for any  \(j\geq 1\) and  \(x>\E[-\log A]\) .
\begin{align}
\P(H_{j}(X)\geq xj)&\leq \P(\max_{0\leq k\leq j}[t \mathscr{S}_{k}-\psi(-t)k]\geq xt j-\psi(-t)j)\nonumber\\
&\leq \P(\max_{0\leq k\leq j}e^{t \mathscr{S}_{k}-\psi(-t)k}\geq e^{(xt -\psi(-t))j})\nonumber\\
&\leq e^{-(xt-\psi(-t))j},\label{vrjpGW-bdX}
\end{align}
which implies that
\begin{equation}
\label{vrjpGW-eq-cramer2}
\P(H_{j}(X)\geq xj)\leq e^{-I(-x)j}.
\end{equation}
Further, for \(0<x<\E[-\log A]\), one sees that by Cram\'er's theorem,
\begin{align}
\P(H_j(X)\leq xj)&\leq \P(X_1+\cdots+X_j\leq xj)\nonumber\\
&=\P(-X_1-\cdots-X_j\geq -xj)\leq e^{-I(-x)j}.\label{vrjpGW-bdX-}
\end{align}
Take \(\eta>0\). In~\eqref{vrjpGW-bdpotentiel}, we can replace \(H_{i}(-X)\wedge H_{n-i-1}(X)\) by \(H_{i}(-X)\wedge H_{n-i-1}(X)\wedge K\eta n\) with some \(K\geq1\) large enough. In fact,
\begin{multline*}
\E[e^{2\lambda H_{i}(-X)\wedge H_{n-i-1}(X)}]
\leq\underbrace{\E[e^{2\lambda H_{i}(-X)\wedge H_{n-i-1}(X)}; H_{i}(-X)\vee H_{n-i-1}(X)\leq K\eta n]}_{\Xi^-_K(i)}\\
+\underbrace{\E[e^{2\lambda H_{i}(-X)\wedge H_{n-i-1}(X)};H_{i}(-X)\vee H_{n-i-1}(X)\geq K\eta n ]}_{=:\Xi^+_K(i)}.
\end{multline*}
Observe that
\begin{align*}
\Xi^+_K(i)&\leq \E[e^{2\lambda H_{i}(-X)};H_{i}(-X)\geq K\eta n ]+\E[e^{2\lambda  H_{n-i-1}(X)};H_{n-i-1}(X)\geq K\eta n ]\\
&=:\Xi_1+\Xi_2
\end{align*}
Let us bound \(\Xi_{1}\),
\begin{align*}
\Xi_1=&\E\int_{-\infty}^{H_i(-X)}2\lambda e^{2\lambda x}{\bf 1}_{H_i(-X)\geq K\eta n}dx=\int_{\mathbb{R}} 2\lambda e^{2\lambda x}\P(H_i(-X)\geq K\eta n\vee x)dx\\
=&\int_{-\infty}^{K\eta n}2\lambda e^{2\lambda x}dx\P(H_i(-X)\geq K\eta n)+\int_{K\eta n}^\infty 2\lambda e^{2\lambda x}\P(H_i(-X)\geq x)dx\\
=& e^{2\lambda K\eta n}\P(H_i(-X)\geq K\eta n)+\int_K^\infty 2\lambda \eta n e^{2\lambda t\eta n}\P(H_i(-X)\geq t\eta n)dt
\end{align*}
By applying~\eqref{vrjpGW-bd-X}, one sees that for any \(0\leq i\leq n-1\) and \(\mu=3>2\lambda\),
\begin{align*}
\Xi_1\leq& e^{2\lambda K\eta n} e^{-\mu K\eta n+\psi(\mu)i}+\int_K^\infty 2\lambda \eta n e^{2\lambda t\eta n} e^{-\mu t\eta n+\psi(\mu)i}dt\\
\leq & e^{-K\eta n+\psi(3)n}+2\lambda e^{\psi(3)n}\int_K^\infty \eta n e^{-t\eta n}dt\\
\leq & 3e^{-K\eta n+\psi(3)n},
\end{align*}
which is less than 1 when we choose \(K\) large enough. Similarly, we can show that for any \(i\leq n-1\),
\[
\Xi_2\leq 1,
\]
for \(K\) large enough. Consequently,~\eqref{vrjpGW-bdpotentiel} becomes that
\begin{equation}
\E\widetilde{E}^\omega_{0}[\tilde{\tau}_{-1}\wedge \tilde{\tau}_{n}]^\lambda\lesssim_{n} 3n^3+n^3\max_{0\leq i\leq n-1}\Xi^-_K(i).
\end{equation}
It remains to bound \(\Xi^-_K(i)\). Take sufficiently small \(\varepsilon>0\) and let \(L=\lfloor\frac{1}{\varepsilon}\rfloor\). For any \(i\) such that \(l_1\lfloor \varepsilon n\rfloor \leq i<(l_1+1)\lfloor \varepsilon n\rfloor\) and \(l_2\lfloor \varepsilon n\rfloor\leq n-i-1<(l_2+1)\lfloor \varepsilon n\rfloor\) with \(0\leq l_1, l_2\leq L\), we have
\begin{align*}
\Xi^-_K(i) \leq & \sum_{0\leq k_1,k_2\leq K}e^{2\lambda k_1\wedge k_2 \eta n+2\lambda \eta n}\P(k_1\eta n\leq H_{i}(-X)< (k_1+1)\eta n)\P(k_2\eta n\leq H_{n-i-1}(X)<(k_2+1)\eta n)\\
\leq & \sum_{0\leq k_1,k_2\leq K}e^{2\lambda k_1\wedge k_2 \eta n+2\lambda \eta n}\P( H_{i}(-X)\geq k_1\eta n)\P(k_2\eta n\leq H_{n-i-1}(X)<(k_2+1)\eta n).
\end{align*}
By~\eqref{vrjpGW-eq-cramer1}, we have
\begin{align*}
  \P( H_{i}(-X)\geq k_1\eta n)\leq e^{-I(x_{1})i}
\end{align*}
where \(x_1\) is the point in \([\frac{k_1\eta n}{(l_1+1)\lfloor \varepsilon n\rfloor}, \frac{k_1\eta n}{l_1\lfloor \varepsilon n\rfloor}]\) where \(I\) reaches the minimum in this interval. By large deviation estimates~\eqref{vrjpGW-eq-cramer2}~\eqref{vrjpGW-bdX-}, we have
\begin{align*}
  \P(k_2\eta n\leq H_{n-i-1}(X)<(k_2+1)\eta n)\leq e^{-I(x_{2})(n-i)}
\end{align*}
where \(x_2\) is the point in \([\frac{k_1\eta n}{(l_2+1)\lfloor \varepsilon n\rfloor}, \frac{(k_2+1)\eta n}{l_2 \lfloor \varepsilon n\rfloor}]\) where \(I\) reaches the minimum in this interval. Therefore,
\begin{align*}
\Xi^-_K(i)\leq & \sum_{0\leq k_1,k_2\leq K} e^{2\lambda k_1\wedge k_2 \eta n+2\lambda \eta n} e^{-I(x_1)l_1\lfloor \varepsilon n\rfloor}e^{-I(-x_2)l_2\lfloor \varepsilon n\rfloor}
\end{align*}
 Taking maximum over all \(l_1,l_2,k_1,k_2\) yields that
\begin{equation}
\E\widetilde{E}^\omega_{0}[\tilde{\tau}_{-1}\wedge \tilde{\tau}_{n}]^\lambda\lesssim_{n} 3n^2+n^2 K^2 \max_{l_1,l_2,k_1,k_2} \exp\{2\lambda k_1\wedge k_2 \eta n+2\lambda \eta n-I(x_1)l_1\lfloor \varepsilon n\rfloor-I(-x_2)l_2\lfloor \varepsilon n\rfloor\}.
\end{equation}
Observe that
\begin{align*}
&2\lambda k_1\wedge k_2 \eta n+2\lambda \eta n-I(x_1)l_1\lfloor \varepsilon n\rfloor-I(-x_2)l_2\lfloor \varepsilon n\rfloor\\
\leq & 2\lambda (x_1l_1\wedge x_2 l_2)\lfloor \varepsilon n\rfloor-I(x_1)l_1\lfloor \varepsilon n\rfloor-I(-x_2)l_2\lfloor \varepsilon n\rfloor+3\lambda \eta n.
\end{align*}
Define
\[
L(\lambda):=\sup_{\mathcal{D}}\{\Big(x_1z_1\wedge x_2z_2\Big)\lambda-I(x_1)z_1-I(-x_2)z_2\},
\]
where \(\mathcal{D}:=\{x_1,x_2,z_1,z_2\geq0, z_1+z_2\leq 1\}\).

By Lemma 8.1 in~\cite{aidekon2008transient}, one concludes that
\[
\limsup_{n\rightarrow\infty}\frac{\log \E\widetilde{E}^\omega_{0}[\tilde{\tau}_{-1}\wedge \tilde{\tau}_{n}]^\lambda}{n}\leq L(2\lambda)=\psi(\frac{1+2\lambda}{2}).
\]
\end{proof}

\section{Some observations on random walks on random trees}
\label{vrjpGW-rwreAppendix}
\begin{proof}[Proof of Lemma~\ref{vrjpGW-bdT}]
As \(\beta(x)\) is identically distributed under \(\p\),
  \begin{align*}
    \e_{\rho}(\sum_{|x|=n}\mathds{1}_{\tau_x<\infty})\e(\beta)&=\e[\sum_{|x|=n}P_{\rho}^{\omega,T}(\tau_x<\infty)]\e(\beta)\\
&=\e \left( \sum_{|x|=n}\e^T(P_{\rho}^{\omega,T}(\tau_x<\infty))\e^T(\beta(x))\right).
  \end{align*}
Here we used the fact that \(\e^{T}{P_{\rho}^{\omega,T}(\tau_{x}<\infty)}\) and \(\e^{T}{(\beta(x))}\) are independent. Now \(P_{\rho}^{\omega,T}(\tau_x<\infty)\) is an increasing function of \(A_x\) since
\begin{align*}
  P_{\rho}^{\omega,T}(\tau_x<\infty)&=P_{\rho}^{\omega,T}(\tau_{\parent{x}}<\infty)\left(\sum_{k\geq 0}P_{\parent{x}}^{\omega,T}(\tau_{\parent{x}}^*<\min(\tau_x,\infty))^k\right)p(\parent{x},x)\\
&=\frac{P_{\rho}^{\omega,T}(\tau_{\parent{x}}<\infty)}{1-P_{\parent{x}}^{\omega,T}(\tau_{\parent{x}}^*<\min(\tau_x,\infty))}\frac{A_{\parent{x}}A_x}{1+A_{\parent{x}}B_{\parent{x}}},
\end{align*}
recall that \(\beta(x)\) is also an increasing function of \(A_x\), moreover, conditionally on \(A_x\), \(P_{\rho}^{\omega,T}(\tau_x<\infty)\) and \(\beta(x)\) are independent, thus by FKG inequality,
\begin{align*}
  \e^T(P_{\rho}^{\omega,T}(\tau_x<\infty)\beta(x))&=\e^T( \e^T(P_{\rho}^{\omega,T}(\tau_x<\infty)\beta(x)|A_x ))\\
&=\e^T(\e^T(P_{\rho}^{\omega,T}(\tau_x<\infty)|A_x)\e^T(\beta(x)|A_x) )\\
&\geq \e^T(P_{\rho}^{\omega,T}(\tau_x<\infty))\e^T(\beta(x)).
\end{align*}
Therefore,
\begin{align*}
  \e \left( \sum_{|x|=n}\e^T(P_{\rho}^{\omega,T}(\tau_x<\infty))\e^T(\beta(x))\right)&\leq \e \left(\sum_{|x|=n}\e^T(P_{\rho}^{\omega,T}(\tau_x<\infty)\beta(x))\right)\\
&=\e\left(\sum_{|x|=n}P_{\rho}^{\omega,T}(\tau_x<\infty)\beta(x)\right).
\end{align*}
For any GW tree and any trajectory on the tree, there is at most one regeneration time at the \(n\)-th generation, therefore,
\[\sum_{|x|=n}\mathds{1}_{\tau_x<\infty,\ \eta_k\neq \parent{x},\forall k>\tau_x}\leq 1.\]
By taking expectation w.r.t.\ \(E_{\rho}^{\omega,T}\) and using the Markov property at \(\tau_x\),
\[\sum_{|x|=n}P_{\rho}^{\omega,T}(\tau_x<\infty)\beta(x)\leq 1.\]
Whence
\[ \e(\sum_{|x|=n}\mathds{1}_{\tau_x<\infty})\e(\beta)\leq 1.\]
By transient assumption it suffices to take \(c_{11}=\frac{1}{\e(\beta)}<\infty\).
\end{proof}
\begin{proof}[Proof of Lemma~\ref{vrjpGW-lem:beta} and Corollary~\ref{vrjpGW-coro-1}]
  Let \(T_i,\ i\geq 1\) be independent copies of GW tree with offspring distribution \((q)\), each endowed with independent environment \((\omega_x,x\in T_i)\). Let \(\rho^{(i)}\) be the root of \(T_i\). In such setting, \(\beta(\rho^{(i)}),\ i\geq 1\) are i.i.d.\ sequence with common distribution \(\beta\).

For each \(T_i\), take the left most infinite ray, denoted \(v_0^{(i)}=\rho^{(i)},v_1^{(i)},\ldots,v_n^{(i)},\ldots\) Let \(\Omega(x)=\{y\neq x;\ \parent{x}=\parent{y}\}\) be the set of all brothers of \(x\). Fix some constant \(C\), define
\[R_i=\inf\{n\geq 1;\ \exists z\in\Omega(v_n^{(i)}),\ \frac{1}{A_z\beta(z)}\leq C\}.\]
By Equation~\eqref{vrjpGW-beta},
\[\frac{1}{\beta(v_{R_{i-1}}^{(i)})}\leq 1+\frac{1}{A_{v_{R_i-1}^{(i)}}A_z\beta(z)}\leq 1+\frac{C}{A_{v_{R_i-1}^{(i)}}}.\]
Also \(R_i\) and \(\{A_{v_n^{(i)}},\ n\geq 0\}\) are independent under \(Q\). By iteration,
\begin{align*}
  \frac{1}{\beta(\rho^{(i)})}&\leq 1+\frac{1}{A_{v_0^{(i)}}A_{v_1^{(i)}}\beta(v_1^{(i)})}\leq 1+\frac{1}{A_{v_0^{(i)}}A_{v_1^{(i)}}}(1+\frac{1}{A_{v_1^{(i)}}A_{v_2^{(i)}}\beta(v_2^{(i)})})\\
&\leq \cdots\\
&\leq 1+\sum_{k=1}^{R_i-1}\frac{1}{A_{v_0^{(i)}}A_{v_k^{(i)}}}\prod_{j=1}^{k-1}A_{v_j^{(i)}}^{-2}+\frac{C}{A_{v_0^{(i)}}}\prod_{l=1}^{R_i-1}A_{v_l^{(i)}}^{-2}.
\end{align*}
For any \(n\geq 0\), denote
\begin{equation}\label{vrjpGW-Cn}
\mathcal{C}(n)=1+\sum_{k=1}^{n}\frac{1}{A_{v_0^{(i)}}A_{v_k^{(i)}}}\prod_{j=1}^{k-1}A_{v_j^{(i)}}^{-2}+\frac{C}{A_{v_0^{(i)}}}\prod_{l=1}^{n}A_{v_l^{(i)}}^{-2}.
\end{equation}
Thus \(\displaystyle \frac{1}{\beta(\rho^{(i)})}\leq \mathcal{C}(R_i-1)\), note also that, since \(\xi_{2}=\E(A^{-2})=1+\frac{3}{c^2}+\frac{3}{c^4}\), \(E(\mathcal{C}(n))\leq c_{34}\xi_2^{n+1}\). Therefore, for any \(K\geq 1\),
\[\frac{1}{\sum_{i=1}^K\beta(\rho^{(i)})}\leq \mathcal{C}(\min_{1\leq i\leq K}R_i-1).\]
Taking expectation under \(\p\) yields (as \(R_i\) i.i.d.\ let \(R\) be a r.v.\ with the common distribution)
\begin{align*}
  \e(\frac{1}{\sum_{i=1}^K\beta(\rho^{(i)})})&\leq \e( \e(\mathcal{C}(\min_{1\leq i\leq K}R_i-1) | R_i;\ 1\leq i\leq K ))\\
&\leq c_{34}\e(\xi_2^{\min_{1\leq i\leq K}R_i})\leq c_{34}\sum_{n=0}^{\infty}\xi_2^{n+1} \p(R \geq n+1)^K\\
&\leq c_{34} \sum_{n\geq 0}\xi_2^{n+1}\e(\delta_C^{\sum_{k=0}^{n-1}(d(v_k)-2)})^K
\end{align*}
where \(\delta_C=\p(\frac{1}{A_{\rho}\beta_{\rho}}>C)\). Let \(f(s)=\sum_{k\geq 1}q_ks^k\), as \(f(s)/s\downarrow q_1\) as \(s\downarrow 0\), for any \(\varepsilon>0\), we can take \(C\) large enough to ensure \(\frac{f(\delta_C)}{\delta_C}\leq q_1(1+\varepsilon)\), thus
\begin{align*}
  \e(\frac{1}{\sum_{i=1}^K\beta(\rho^{(i)}) })&\leq c_{34}\sum_{n\geq 0}\xi_2^{n+1}(\frac{f(\delta_C)}{\delta_C})^{nK}\leq c_{34}\sum_{n\geq 0}\xi_2^{n+1} (q_1(1+\varepsilon))^{nK}.
\end{align*}
Now take \(\varepsilon\) such that \(q_1(1+\varepsilon)<1\), then take \(K\) large enough such that \(\xi_2(q_1(1+\varepsilon))^{K}<1\) leads to
\[\e(\frac{1}{\sum_{i=1}^K\beta(\rho^{(i)}) })<c_{12}<\infty\]
Similarly, the following also holds
\[\e(\frac{1}{\sum_{i=1}^KA_{\rho^{(i)}}\beta(\rho^{(i)}) })<c_{12}<\infty.\]
In particular, if \(q_1\xi_2<1\), we can take \(K=1\) in \(\xi_2(q_1(1+\varepsilon))^{K}<1\).
Further, it follows from~\eqref{vrjpGW-Cn} and Chauchy-Schwartz inequality that
\[
\mathcal{C}(n)^2\leq  (n+2)\bigg(1+\sum_{k=1}^{n}\frac{1}{A^2_{v_0^{(i)}}A^2_{v_k^{(i)}}}\prod_{j=1}^{k-1}A_{v_j^{(i)}}^{-4}+\frac{C}{A_{v_0^{(i)}}}\prod_{l=1}^{n}A_{v_l^{(i)}}^{-4}\bigg).
\]
Thus,
\[
\e[\mathcal{C}^2(n)]\leq c_{35}(n+2)\xi_4^{n+1}.
\]
As soon as \(\xi_4<\infty\), the previous argument works again to conclude that for \(K\) large enough,
\[
\e(\frac{1}{\sum_{i=1}^{K}\beta^2(\rho^{(i)}) })+\e(\frac{1}{\sum_{i=1}^{K_0}A^2_{\rho^{(i)}}\beta^2(\rho^{(i)}) })<c_{13}<\infty.
\]
\end{proof}

{\bf \noindent Acknowledgments: }We would like to thank an anonymous referee for carefully reading the paper and providing
corrections.
\nocite{*}
\bibliography{bibi}{}
\bibliographystyle{plain}
\end{document}